\documentclass[11pt,reqno]{amsart}

\usepackage{graphicx}
\usepackage{latexsym}
\usepackage{amssymb}
\usepackage{amsmath}
\usepackage{enumerate}

\usepackage{amsmath}
\usepackage{stmaryrd}
\usepackage{hyperref}
\usepackage{amssymb}
\usepackage{CJK}
\usepackage{color}

\usepackage{amsmath,latexsym,amssymb,amsthm,array,amsfonts}
\usepackage{mathtools, stmaryrd}
\usepackage{xparse} \DeclarePairedDelimiterX{\Iintv}[1]{\llbracket}{\rrbracket}{\iintvargs{#1}}
\NewDocumentCommand{\iintvargs}{>{\SplitArgument{1}{,}}m}
{\iintvargsaux#1} %
\NewDocumentCommand{\iintvargsaux}{mm} {#1\mkern1.5mu..\mkern1.5mu#2}
\usepackage{mathrsfs,dsfont,bbm}
\usepackage{csquotes}
\usepackage{layout}
\usepackage{layout}

\newtheorem{lemma}{Lemma}
\newtheorem{definition}{Definition}
\newtheorem{corollary}{Corollary}
\newtheorem{theorem}{Theorem}
\newtheorem{proposition}{Proposition}
\newtheorem{remark}{Remark}
\newtheorem{example}{Example}
\newtheorem{conjecture}{Conjecture}

\def\real{{\mathord{{\rm I\kern-2.8pt R}}}}        
\def\inte{{\mathord{{\rm I\kern-2.8pt N}}}}

\def\Dom{{\mathrm{{\rm Dom}}}}

\def\sZZ{{\rm Z\kern-2.8ptem{}Z}}

\def\z{{\mathchoice
  {\sZZ}
  {\sZZ}
  {\rm Z\kern-0.30em{}Z}
  {\rm Z\kern-0.25em{}Z} }}
\def\sQQ{{\kern 0.27em \vrule height1.45ex width0.03em depth0em
          \kern-0.30em \rm Q}}
\def\qu{{\mathchoice
    {\sQQ}
    {\sQQ}
  {\kern 0.225em \vrule height1.05ex width0.025em depth0em \kern-0.25em \rm Q}
  {\kern 0.180em \vrule height0.78ex width0.020em depth0em \kern-0.20em \rm Q}
        }}
\def\sCC{{\kern 0.27em \vrule height1.45ex width0.03em depth0em
          \kern-0.30em \rm C}}
\def\complex{{\mathchoice
    {\sCC}
    {\sCC}
  {\kern 0.225em \vrule height1.05ex width0.025em depth0em \kern-0.25em \rm C}
  {\kern 0.180em \vrule height0.78ex width0.020em depth0em \kern-0.20em \rm C}
        }}

\MakeOuterQuote{"}

\newcommand{\ba}{\begin{array}}
\newcommand{\ea}{\end{array}}
\newcommand{\be}{\begin{equation}}
\newcommand{\ee}{\end{equation}}
\newcommand{\bea}{\begin{eqnarray}}
\newcommand{\eea}{\end{eqnarray}}
\newcommand{\beaa}{\begin{eqnarray*}}
\newcommand{\eeaa}{\end{eqnarray*}}

\def\z{\zeta}

\font\tenmath=msbm10 \font\sevenmath=msbm7 \font\fivemath=msbm5
\newfam\mathfam \textfont\mathfam=\tenmath
\scriptfont\mathfam=\sevenmath \scriptscriptfont\mathfam=\fivemath

\def \={{\buildrel {\rm (law)} \over =}}

\def\qed{ \hfill \vrule width.25cm height.25cm depth0cm\smallskip}

\newcommand{\basa}{\begin{assumption}}
\newcommand{\easa}{\end{assumption}}

\newcommand{\bas}{\begin{assum}}
\newcommand{\eas}{\end{assum}}

\usepackage{amsmath,amsthm}
\usepackage{amssymb}
\DeclareMathOperator{\Hess}{Hess}

\DeclareMathOperator{\Vol}{Vol}
\DeclareMathOperator{\const}{const}
\DeclareMathOperator{\Ric}{Ric}
\DeclareMathOperator{\supp}{supp}

\DeclareMathOperator{\Tr}{Tr}

\DeclareMathOperator{\Id}{Id}

\newcommand{\ignore}[1]{}
\begin{document}

\title{Free Stein kernels and moment maps }

\author{Charles-Philippe Diez$^{\dagger}$}
\thanks{$^{\dagger}$ University of Luxembourg, Department of Mathematics, Maison du Nombre,
6 avenue de la Fonte, L-4364 Esch-sur-Alzette, Grand Duchy of Luxembourg, \textbf{charles-philippe.diez@uni.lu}}

\begin{abstract}
In this paper, we propose a free analogue to Fathi's construction of Stein kernels using
moment maps (2019). This is possible for a class of measures called free moment measures that was introduced in the free case by Bahr and Boschert (2021), and by using the notion of free moment maps which are
convex functions, solutions of a variant of the free Monge-Ampère equation discovered by Guionnet
and Shlyakhtenko (2012). We then show how regularity estimates in some weighted non-commutative
Sobolev spaces on these maps control the transport distances to the semicircular law. We also prove
in the one dimensional case a free analogue of the moment map version of the Cafarelli contraction
theorem (2001), discovered by Klartag in the classical case (2014), and which leads to a uniform bound on the free moment Stein kernel. Finally, we discuss the applications of these results: we prove a stability
result characterizing the semicircular distribution among a certain subclass of free Gibbs measures,
the probabilistic interpretation of this free moment Stein kernel in terms of free diffusion processes, its connections with the theory of non-commutative Dirichlet forms, and a possible notion of a
non-commutative (free) Hessian manifold associated with a free Gibbs measure, conjecturally having free analogue properties of a classical Hessian manifold associated with a log-concave measure
and considered by Kolesnikov (2012), which has the striking property of having a Ricci curvature bounded from below by $1/2$.
\end{abstract}
\maketitle

\section{Introduction}
\bigbreak
Fathi, in his important paper \cite{mm}, discovered a new way to implement Stein kernels with respect to the standard Gaussian measure via ideas of optimal transport. In fact, there is a variant of the Monge-Kantorovitch problem, where, instead of specifying both measures and looking for a transport map that sends one to the other, we fix a target measure $\mu$ and look for an essentially continuous convex function $\varphi$, called the "{\it moment map}", such that $\mu$ is the pushforward of the probability measure with density $e^{-\varphi}$ by the gradient of $\varphi$, that is $\mu=(\nabla \varphi)_{\sharp} e^{-\varphi}dx$. This formal construction (which is a variant of the Minkowski problem for encoding the convex sets by measures on the unit sphere) can thus be seen as a canonical bijection between convex functions and Borel measures on $\mathbb{R}^n$. The existence of such functions $\varphi$ was proved by Cordero-Erausquin and Klartag in \cite{CEK} for a large class of measures $\mu$ called "{\it moment measures}" which have their barycenter at the origin (and thus have a finite first moment) and are not supported on a hyperplane. This is an important extension which was originally studied for complex geometry purposes, e.g. the moment maps associated to convex polytopes play a important role in the construction of K\"ahler-Einstein metrics on toric varieties, see Berman and Berndtsson \cite{BerB} (K\"ahler-Einstein metrics on log-Fano varieties), Wang and Zhu \cite{wang}, Donaldson \cite{DON} or Legendre \cite{legendre}, and recently extended by Rotem \cite{Rotem1,Rotem2} using surfaces area measure of log-concave functions.
\bigbreak
From this construction, Fathi in \cite{mm} realized through a fixed point argument that the \textit{moment map} contains a lot of information on how a target measure is close to the standard gaussian one and realised that it is possible to construct Stein kernels using these {\it moment maps}, which are weak solutions of a {\it toric Kähler-Einstein} equation, thus giving a new powerful tool to compare $\mu$ with the standard Gaussian measure $\gamma$, and to prove that suitable regularity estimates in some weighted Sobolev spaces on these moment maps control the rate convergence to the standard Gaussian law. He was also able to generalise his construction of Stein kernels with respect to more general reference measures, i.e. for $e^{-V}dx$, with $\Hess\: V>0$, and was able to obtain new (and also dimensionally sharp) estimates for the CLT in $L^p, p\geq 2$ Wasserstein distances for uniformly log-concave isotropic probability measures by using the breakthrough work of Klartag in \cite{KlaKol} which obtained universal regularity estimates for the \textit{moment map} in some weighted Sobolev spaces for the class of log-concave measures associated with a smooth uniformly convex potential.
\bigbreak
    Motivated by the role played by the semicircular distribution in the free context, the aim of this paper is to study the free counterparts of Fathi's results in order to translate these results into the free world, where we will see that free moment Stein kernels exist in great generality in the one-dimensional case and in a perturbative regime in the multidimensional case. We will also show how regularity estimates on a variant of the free Monge-Ampère equation of Guionnet and Shlyakhtenko \cite{GS}, called the free Kähler-Einstein equation, control the transport distance to the semicircular law. We will then use Caffarelli's maximum principle argument \cite{Cafa} to prove a free Klartag's contraction theorem in the one-dimensional case: for a centered free Gibbs measure associated with a smooth $\epsilon$-uniformly convex potential, if its \textit{moment map} is smooth enough, we will show that the \textit{moment map} is then $\epsilon^{-1}$-Lipschitz; which we will straightforwardly adapt to Cafarelli's free version (sending the semicircular law onto another smooth free Gibbs measure). This will lead to free moment Stein kernels uniformly bounded by the inverse of the convexity constant. More interestingly, we will introduce a potential underlying structure analogous to the weighted Riemannian manifold equipped with a Hessian metric given by the Riemmanian metric tensor $g=\Hess \varphi$, i.e. $M^*_{\varphi}:=(\mathbb{R}^n, \Hess \varphi, e^{-\varphi}dx)$ and generally called a \textit{Hessian manifold}, which is encoded by a smooth convex function: $\varphi:\mathbb{R}^n\rightarrow \mathbb{R}$, and which has been studied intensively by Klartag and Kolesnikov for optimal transport and convex geometry purposes \cite{Klart,Kol,KlaKol,Kol2}. It is known that this manifold is the "realisation" (the opposite of "complexification" according to Klartag) of a certain toric K\"ahler manifold (even true in the orbifold case which is a manifold with singularities), which has moreover spectacular curvature bounds since the Ricci tensor associated with the K\"ahler form induced by $\varphi$ will be exactly equal to half the metric tensor.  This non-commutative structure will be established in full generality in the one-dimensional case and again for perturbations close to the semicircular law in the multidimensional case.  We will see that an important property, namely, that the derivatives of the moment map in any direction is an eigenfunction of a certain free diffusion operator, or more precisely the free Laplacian on what we will call a non-commutative Hessian manifold, which is always associated with the eigenvalue "$-1$" and contains all linear functionals as in the classical case, see Klartag and Kolesnikov \cite{KlaKol}, equation $(2.1)$ for the classical results. We then introduce a (bi-modular) \textit{carré-du-champ} operator and conjecture that the Ricci curvature of this non-commutative Hessian manifold is bounded from below by $1/2$, which would be the free analogue of the result proved by Kolesnikov in \cite{Kol}. We will also explore connections with other branches of free probability, such as invariant measures of free diffusion processes, or with the theory of non-commutative Dirichlet forms and free spectral gap (Poincaré) inequalities. Finally, since the main goal of the free monotone transport of Guionnet and Shlyakhtenko \cite{GS} was to prove isomorphism between von Neumann algebras, we explain why, at least heuristically (for a large number of generators), free Gibbs laws around the semicircular laws generate von Neumann algebras which are all isomorphic to the free group factor (Guionnet and Shlyakhtenko \cite{GS}): this seems to be the free analogue of Klartag's CLT for convex bodies \cite{convex}, which has a nice interpretation in the moment map setting that we will translate into the free world. Thus enforcing in the free context, that convexity can "\textit{replace the role of independence in certain aspects of the phenomenon represented by the central limit theorem}": Klartag \cite{convex}.
\newline
Moreover, if the moment map construction of Bahr and Boschert \cite{BB} could be extended beyond this even perturbative setting, it would also give a heuristic and a possible way of understand better how von Neumann algebra generated by free Gibbs laws associated to a regular convex potential are shown to be isomorphic to a free group factor $L(\mathbb{F}_n)$ as proved by Dabrowski, Guionnet, Shlyakhtenko and Jekel \cite{YGS,jektriang}.
\section{Definitions and Notations}
\bigbreak
In the following we consider a self-adjoint operator $x\in \mathcal{M}$ (a finite von Neumann algebra equipped with $\tau$ a faithful normal tracial state) with analytic distribution $\mu:=\mu_x=\tau\circ E_x$, where $E_x$ denotes the spectral measure of $x$ on its spectrum $\sigma(x)\subset \mathbb{R}$.

\bigbreak
\begin{definition}
A free Stein kernel for a self-adjoint operator $x\in (\mathcal{M},\tau)$ with analytic distribution $\mu:=\mu_X$ relative to a $\mathcal{C}^1$ potential $u:\mathbb{R}\rightarrow \mathbb{R}$ is a function $A:\mathbb{R}^2\rightarrow \mathbb{R}$ such that for all smooth test functions $f$:
\begin{equation}
\int u'(x)f(x)d\mu(x)=\iint A(x,y)\mathscr{J}f(x,y)d\mu(x)d\mu(y),
\end{equation}
The Stein discrepancy of $x$ relative to $u$ is then defined as 
\begin{equation}
\Sigma ^*(x|u)=\inf_{A}\iint\bigg[A(x,y)-1\bigg]^2d\mu(x)d{\mu}(y),
\end{equation}
where the infinimum is taken over all admissible Stein kernels $A$ of $x$ relative to $u$.
\end{definition}
If the measure $\mu$ is non-atomic, which is known from Voiculescu's work (Appendix in \cite{Vca}) to be equivalent to $\mathscr{J}f$, which is closable as a densely defined operator from $L^{2}(\mathbb{R},d\mu)$ to $L^2(\mathbb{R}^2,d\mu\otimes d\mu)$, the diagonal: $\Delta\subset \mathbb{R}^2$ has measure zero with respect to $\mu\otimes \mu$, and we can relax the Stein equation to
\begin{equation}
\int u'(x)f(x)d\mu(x)=\iint A(x,y)\frac{f(x)-f(y)}{x-y}d\mu(x)d\mu(y),
\end{equation}

\begin{definition}
The standard semicircular distribution is the probability distribution:
\begin{equation}
    \eta(dx)=\frac{1}{2\pi}\sqrt{4-x^2}dx, \; \lvert x\rvert\leq 2,
\end{equation}
This distribution has all its odd moments equal to zero. Its even moments can be expressed as {\it Catalan numbers} thanks to the following relation, valid for all non-negative integers $m$:
\begin{equation}
    \int_{-2}^{2}x^{2m}\eta(dx)=C_m,
\end{equation}
where $C_m=\frac{1}{m+1}\binom{2m}{m}$ is the {$m$-th Catalan number}.
\end{definition}
 
\begin{proposition}(Voiculescu, \cite{Ventro})  
    Let $\mu$ be a Borel probability measure on $\mathbb{R}$, then its free entropy is given by
    \begin{equation}
        \chi(\mu)=\iint \log\lvert t-s\rvert d\mu(s)d\mu(t)+\frac{3}{4}+\frac{1}{2}\log 2\pi,
    \end{equation}
    In particular, and up to constants, the free entropy is the minus sign of the \it{logarithmic energy} of $\mu$.
\end{proposition}
\begin{flushleft}
As in the classical case, Free Gibbs laws are characterised by a minimization problem: they minimize the relative microstate variant of entropy.
\end{flushleft}
\begin{definition}(\cite{GM,FN,VF})
The free Gibbs law $\tau_V$ associated with the potential $V$, if it exists, is the minimizer of the functional:
\begin{equation}
\chi_V(\tau):=-\chi(\tau)+\tau(V),
\end{equation}
\end{definition}
\begin{flushleft}
In the one-dimensional case, existence and uniqueness of a minimizer is ensured when $V$ is lower semi-continuous and satisfies 
\begin{equation}\label{gibbs}
    \underset{\lvert x\rvert \rightarrow \infty}{\lim}(V(x)-2\log\lvert x\rvert)=+\infty,
\end{equation}
It is known (see \cite{Saff} or \cite{LP}) that under these assumptions that there exists a unique minimizer $\nu_V$ in the set $\mathscr{P}(\mathbb{R})$ and the solution is compactly supported.
\end{flushleft}
\begin{remark}
    It is known that the existence of a minimizer is ensured in all dimensions if $V$ is a small analytic perturbation of the quadratic potential. 
\newline
We also know that if $V$ is bounded from below, satisfies a growth condition at infinity and a H\"older-type continuous criterion \cite{DmV}, we also have uniqueness of the minimizer in \eqref{gibbs}.
\end{remark}

In particular, and focusing only on convex potentials $u:\mathbb{R}\rightarrow \mathbb{R}$ (and thanks to Voiculescu's explicit calculations), we recall that in the one-dimensional case, the free entropy is the opposite of the logarithmic energy (up to an explicit constant), thus we are left to minimize the following functional.
\begin{definition}
The free Gibbs measure $\nu_u$ associated with the convex function $u:\mathbb{R}\rightarrow \mathbb{R}$ is the measure corresponding to the free Gibbs law $\tau_u$. This is the unique minimizer of the functional
\begin{equation}
-\iint \log\lvert t-s\rvert d\mu(t)d\mu(s)+\int u(s)d\mu(s),
\end{equation}
\end{definition}
\bigbreak
It was also noticed by Voiculescu (see section 3.7 of \cite{VF}), that the above minimizer $\nu_u$, should satisfy the following {\it Euler-Lagrange} equation: 
\begin{equation}\label{HT}
    2\pi H\nu_u=u',\: \nu_u-a.e
\end{equation}
where $Hp$ denotes the Hilbert transform of a probability measure $\rho$ which is defined as:
\begin{equation}
    H\rho(t)=\frac{1}{\pi}PV\int_{\mathbb{R}}\frac{1}{t-x}d\rho(x),
\end{equation}
Since $\nu_u$ minimizes $\chi_u(\cdot)$. We have for a test function $f$, and $\delta\in \mathbb{R}$, that
\begin{equation}
    \chi_u((x+\delta f)_{\sharp} \nu_u)\leq \chi_u(\nu_u),
\end{equation}
Then, by considering $\frac{d}{dt}{\chi_u((x+\delta f)_{\sharp}\nu_u)}_{|t=0}$, it is implied that this last equation \eqref{HT} can also be rewritten by a Schwinger-Dyson (type) equation as follows (which is also true in the multidimensional case, using the microstate version of free entropy).
\begin{proposition}(Guionnet and Shlyakhtenko: theorem 5.1 in \cite{GS}, Fathi and Nelson: lemma 1.6 in \cite{FN})
The unique minimizer of the previous functional  is necessarily a solution of the following Schwinger-Dyson equation:
\begin{equation}
    \int u'(x)f(x)d\nu_u(x)=\iint \frac{f(x)-f(y)}{x-y}d\nu_u(x)d\nu_u(y),
\end{equation}
for all nice test functions $f$.
\end{proposition}
\begin{remark}\label{rem}
Guionnet and Maurel-Ségala in \cite{GM} have indeed proved the equivalence between the Schwinger-Dyson equation and the above minimisation problem assuming that the free Gibbs measure $\nu_u$ exists, is unique and has a connected support. For example: 
\begin{enumerate}
    \item If $u$ is strictly convex on a sufficiently large interval, i.e. $\exists \kappa>0$, such that $u''(x)\geq \kappa $ for all $x\in \mathbb{R}$. This also ensures that $\nu_u$ is supported on a compact interval and has a density w.r.t. the Lebesgue measure.
    \item The other case where existence and uniqueness (even in the multidimensional setting) are ensured is when we consider a "small" analytic perturbation of the semicircular potential, i.e. for potentials $V=\frac{1}{2}\lVert x\rVert^2+\beta W$ (where $W$ is a n.c. power series and $\beta$ is small. See e.g. \cite{DmV}).
\end{enumerate}  
\end{remark}
\begin{flushleft}
We recall that in dimension one, for functions $f\in \mathcal{C}^1(\mathbb{R})$, we have the following expression for the free difference quotient (in the sequel, we will rather use the notation $\mathscr{J}$ for the free difference quotient when we are dealing with the {\it measure formulation}):
\begin{equation}
 \mathscr{J}f(x,y)=\begin{cases*}
                    \frac{f(x)-f(y)}{x-y} & if $x\neq y$  \\
                     f'(x) & if $x=y$
                 \end{cases*},
\end{equation} 
We also recall that we have the following chain rule, which is easy to check
\begin{equation}
    \mathscr{J}(f\circ g)(x,y)=\mathscr{J}g(x,y)\cdot \mathscr{J}f(g(x),g(y))\nonumber
\end{equation}
We will also denote $\mathscr{J}f(g)$ or sometimes $\mathscr{J}f\circ g$ as
\begin{equation}\label{chain}
    \mathscr{J}f(g)(x,y):=\mathscr{J}f(g(x),g(y))\nonumber
\end{equation}
The cyclic derivative then coincide with the usual derivative, i.e $\mathscr{D}u(x)=u'(x)$.
\bigbreak
Finally, we will call the non-commutative Hessian in the free context as the operator $\mathscr{J}\mathscr{D}$, and we will denote in a shorthand $(\mathscr{J}f)^{-1}$ as \begin{equation}
    (\mathscr{J}f)^{-1}(x,y):=(\mathscr{J}f(x,y))^{-1}=\frac{x-y}{f(x)-f(y)}
\end{equation}
when this expression makes sense (for example in the sequel we will consider the case $f=u'$ where $u$ is $\mathcal{C}^2$-smooth and strictly convex this expression is licit). \end{flushleft}

\section{Free Stein kernels and moments maps}\label{4}
In this section we will focus only on the one-dimensional (commutative) case, where we have a proper notion of the analytic distribution $\mu_x$ of a non-commutative random variable $x\in (\mathcal{M},\tau)$ (we will focus only on the "measure" version and deal with compactly supported measure instead of non-commutative random variables), we will study the construction of free Stein kernels for compactly supported measure thanks to a notion of {\it free moment maps} developed in the free setting by Bahr and Boschert in \cite{BB}, and which corresponds to a free analogue of a breakthrough result in the classical case by Cordero-Erausquin and Klartag in \cite{CEK}, which gives a large class of probability measures that are the pushforward of a convex function $\varphi$ (the transport map) with respect to the law of density $e^{-\varphi}$. 
\begin{flushleft}
In the sequel, we recall that for $\varphi:\mathbb{R}^d\rightarrow \mathbb{R}\cup\left\{+\infty\right\}$ a function, convex or not, and not identically ${+\infty}$, we denote its Legendre transform, as the convex function:
\begin{equation}
\varphi^*(y)=\sup_{\substack{x\in \mathbb{R}^d\\ \varphi(x)<\infty}}\left\{x\cdot y-\varphi(x)\right\}
,
\end{equation}
And which satisfies the following property:
\begin{enumerate}
\item $\varphi^*$ is always convex and lower semi-continuous.
\item when $\varphi$ is convex and differentiable at the point $x$, we have:
\begin{equation}
    \varphi^*(\nabla \varphi(x))=x\cdot \nabla \varphi(x)-\varphi(x),\nonumber
\end{equation}

\item If $\varphi$ is $\mathcal{C}^2$, $\varphi^*$ is also $\mathcal{C}^2$, and $\nabla \varphi^*$ is the inverse of $\nabla \varphi$, i.e: $\nabla \varphi^*=(\nabla \varphi)^{-1}$.
\item The unique fixed point of the Legendre transform is given by $\frac{{\lVert x\rVert_2}^{2}}{2}$, i.e. the Legendre transform is a symmetry on the space of convex functions around this fixed point.
\end{enumerate}
\end{flushleft}
Building on earlier work by various authors \cite{BerB,DON,legendre,wang}, the following important theorem was established by Cordero-Erausquin and Klartag in \cite{CEK}. 
\begin{theorem} (Cordero-Erausquin and Klartag in \cite{CEK}) \label{th3}
Let $\mu$ be a measure in $\mathscr{P}_1(\mathbb{R}^d)$ with barycenter at the origin, and not supported on an hyperplane. Then there exists an essentially-continuous convex function $\varphi$ (uniquely determined up to translations), such that $\mu$ is the pushforward of the centered probability measure with density $e^{-\varphi}$, by the map $\nabla \varphi$. The function $\varphi$ is called the moment map of $\mu$.
\end{theorem}
Conversely, if $\varphi$ is an essentially-continuous convex function (this condition ensures the validity of the "integration by parts") and also the uniqueness up to translations and can be thought of roughly as follows: $\varphi$ is essentially continuous if and if only it is lower semi-continuous, and the set of discontinuity points of $\varphi$ has zero $\mathcal{H}^{n-1}$ measure, where $\mathcal{H}^{n-1}$ denotes the $(n-1)$-dimensional Haussdorff measure), with $0<\int_{\mathbb{R}^d}e^{-\varphi}dx<+\infty$, its associated {\it moment measure} is not supported on a hyperplane and its barycenter is at the origin (cf. Proposition $1$ in \cite{CEK}).
\newline
The function $\varphi$, solution of theorem \eqref{th3}, may not be smooth in full generality, especially when the {\it moment law} is a combination of Dirac masses (see the remark before theorem 2.2 in \cite{mm} and e.g. 2.5.2 in \cite{BB}). In fact, consider the uniform measure on $\left\{-1,1\right\}$ as a subset of $\mathbb{R}$, i.e. $\mu=\frac{1}{2}\delta_{-1}+\frac{1}{2}\delta_1$, for which we can compute $\varphi(x)=\frac{1}{2}\lvert x\rvert$, which is not smooth at the origin. Another example is given by the uniform measure on the sphere $\mathbb{S}^{d-1}:=\left\{x \in \mathbb{R}^d, \lVert x\rVert=1\right\}$, for which the moment measure is $\varphi(x)=\lVert x\rVert$. Except in dimension one where we can express the moment map in terms of the the target measure $\mu$, there is no explicit formula in higher dimensions. Fortunately, a result due to Berman and Berndtson \cite{BerB} shows how to derive the smoothness of the moment map via Cafarelli regularity theory. Indeed, if $\mu$ is supported on a compact open convex set with a smooth $\mathcal{C}^{\infty}$ density $\rho$ bounded from above and below, then the moment map is finite and $\mathcal{C}^{\infty}$ in the entire $\mathbb{R}^n$. If we only assume that the density is continuous, bounded from above and below, we can indeed deduce that the moment map is at least $\mathcal{C}^2$, which is sufficient for the considerations of Stein's method, see Fathi's remark $3.1$ in \cite{mm}.
\begin{flushleft}
In our context, Bahr and Boschert \cite{BB} were able to prove a free analogue (especially well adapted in dimension one) of this construction by adapting the variational method of Santambrogio \cite{SA} for the classical moment measures. Note also that the condition that the centered measure $\mu$ is not supported on an hyperplane is reduced in this one-dimensional case to the condition: $\mu\neq \delta_0$.
\end{flushleft}
\begin{theorem}(Bahr and Boschert, Theorem 2.5 in \cite{BB})\label{2.5}
Let $\mu\neq \delta_0$, a probability measure in $\mathscr{P}_2(\mathbb{R})$ with barycenter zero, then there exists a convex, lower semi-continuous function $u$, such that $\mu$ is the pushforward of the free Gibbs measure $\nu_u$ by the function $u'$, i.e $\mu=(u')_{\sharp} \nu_u$. The convex function $u$ is called the {\it free moment map} of $\mu$.
\newline
Moreover, $\nu_u$ is absolutely continuous w.r.t the Lebesgue measure, has compact support, and is the unique centered minimizer of the functional 
\begin{equation}
    \mathcal{F}(\rho)=L(\rho)+T(\rho,\mu)
\end{equation}
defined on $\mathscr{P}_2(\mathbb{R})$, where $L(\rho)$ denote the logarithmic energy and
$T(\rho,\mu)$ is the maximal correlation functional:
\begin{equation}
    T(\rho,\mu)=\sup\left\{\int x\cdot y\: d\gamma, \gamma \in d\pi(\rho,\mu)\right\}
\end{equation}
where $\pi(\rho,\mu)$ is the set of transport plans with marginals $\rho$ and $\mu$.
\end{theorem}
\begin{remark}
    It is also easy to deduce since $\mu$ is centered that the derivative of $u$ must changes signs and so that $u(x)\underset{\pm\infty}{\rightarrow}+\infty$.
\end{remark}
\begin{flushleft}
This theorem gives a large class of probability measures which are the pushforward of the free Gibbs measure $\nu_u$ by the function $u'$, and in particular it gives a (complete, with only a small restriction) analogue of Cordero-Erausquin and Klartag's result in the classical case, but only in dimension one. The method of Bahr and Boschert is in fact based on the free analogue of the optimal method transport functional found by Santambrogio in \cite{SA}.
Interestingly, it is indeed possible to adopt the original view of Klartag and Cordero Erausquin \cite{CEK},which is based on a new "above tangent" version of the Prekopá-Leindler inequality; this will be done in another paper and is based on the study of the type of a \textit{free pressure functional} of the Legendre transform of the moment map. 
\bigbreak
Another nice feature of the Bahr and Boschert construction \cite{BB} is that it allows us to derive sharp free transport cost inequalities. This deep result in the classical case, which has known a current deep interest \cite{frade,MikF}, was first witnessed by Fathi in his important paper \cite{FathT} and enlightens a duality between an inequality of convex geometry called the (functional) Santalo inequality and this sharp symmetrized Talagrand inequality. This inequality actually holds in the free case \cite{Di2}, and strengthens the free transport cost inequality discovered by Biane and Voiculescu \cite{BV} for the semicircular distribution.
\end{flushleft}
\bigbreak
More interestingly, and in the classical case, Fathi's paper \cite{mm} discovered a very interesting connection between the {\it moment map} $\varphi$ and the notion of a Stein kernel. In particular, he noticed that the map
\begin{equation}\label{fix}
\mu \mapsto e^{-\varphi}dx,
\end{equation}
where (here $dx$ standard for the Lebesgue measure on $\mathbb{R}^n$ up to a renormalisation constant) has a unique fixed point given by the standard Gaussian measure, for which we recall that the moment map is given by $\frac{\lVert x\rVert_2^2}{2}$. So the moment map already contains a lot of information about how close $\mu$ is to the standard Gaussian measure $\gamma$.
\bigbreak
In particular, and heuristically, if $\varphi(x)\approx \frac{\lVert x\rVert_2^2}{2}$, then we can expect that $\mu \approx \gamma$. Fathi formalised this argument, and has showed that one can construct a Stein kernel thanks to this moment map. It is then easy to bound the quadratic Wasserstein distance to the standard Gaussian by a regularity estimate on the associated Monge-Ampère (K\"ahler-Einstein) PDE thanks to the {\it Wasserstein-Stein discrepancy} inequality of Ledoux, Nourdin and Peccati \cite{LNP}. 
\begin{flushleft}
We can also note from the convexity of $\varphi$ that $\nabla \varphi$ is the {\it Brenier map} sending $e^{-\varphi}dx$ onto $\mu$, and in particular that the moment map $\varphi$ is basically the weak solution of the following variant of the Monge-Ampère equation, called the {\it toric Kähler-Einstein} equation (because it has numerous applications in Kähler geometry, and in particular in the study of differential complex and symplectic structures on toric Fano manifolds, see e.g. Donaldson \cite{DON} for a nice exposition).
\begin{equation}\label{mongA}
e^{-\varphi}=\rho(\nabla \varphi)\det(\Hess \varphi)
\end{equation}
where $\rho$ is the density with respect to the Lebesgue measure of $\mu$, which should be strictly positive on its support. 
\bigbreak
In particular, for $\mu=e^{-\psi}dx$, then $\varphi$ solve (in a weak sense):
\begin{equation}\label{mongB}
    e^{-\varphi}=e^{-\psi(\nabla \varphi)}\det(\Hess \varphi)
\end{equation}
Thus, if we fix $\rho$ as the Gaussian density, then the unique weak solution of the Monge-Ampère PDE \eqref{mongA} is given by $\frac{\lVert x\rVert^2_2}{2}$, which justifies that the Gaussian is the unique fixed point of the map \eqref{fix}.

\begin{theorem}(Fathi, Theorem 2.3 in \cite{mm})\label{3.3}
Let's suppose that $\mu$ has a density $\rho$ w.r.t the Lebesgue measure which is strictly positive on its support, and that the solution $\varphi$ to the PDE \eqref{mongA} is $\mathcal{C}^2$ and supported on the entire space $\mathbb{R}^d$. Then $\Hess\varphi(\nabla \varphi^*)=(\Hess \varphi^*)^{-1}$ is a free Stein kernel for $\mu$ with respect to the standard Gaussian measure.
\end{theorem}
\end{flushleft}
In the free context and in the one-dimensional case, it seems plausible that a similar statement should hold by following the same heuristic:
\newline 
The map $\mu \mapsto \nu_u$ admits a unique fixed point given by the standard semicircular distribution (this will be easily seen in section \ref{sect4}, where we will see that $u$ solves a variant of the free Monge-Ampère equation), and thus, heuristically, if $u\approx \frac{x^2}{2}$, we should expect that $\mu\approx \mathcal{S}(0,1)$.
\bigbreak
In fact, thanks to this {\it free moment map}, we can also construct free Stein kernels with respect to the semicircular potential and indeed quantify this convergence in the quadratic Wasserstein distance $W_2$. It is also interesting to note that it has exactly the same form as in the classical case, up to the replacement of the classical differential operators by their free counterparts.
\newline
We will therefore show that in the free context we have a free Stein kernel of the form $(\mathscr{J}\mathscr{D}u^*)^{-1}=\mathscr{J}\mathscr{D}u(\mathscr{D}u^*)$, where $u^*$ denotes the Legendre transform of the free moment map.
\begin{theorem}\label{mstein}
Let's assume that $\mu$ is centered, absolutely continuous with respect to the Lebesgue measure and supported on a compact interval, and let $u$ be the free moment map of $\mu$ with respect to the free Gibbs measure $\nu_{u}$, i.e. $\mu=(u')_{\sharp}\nu_{u}$. If $u$ is $\mathcal{C}^2$ and strictly convex on $\mathbb{R}$, then
\begin{eqnarray}
(x,y)\mapsto (\mathscr{J}\mathscr{D}u^*(x,y))^{-1}&=&\mathscr{J}\mathscr{D}u(\mathscr{D}u^*(x),\mathscr{D}^*u(y))
\nonumber\\
&:=&\frac{x-y}{(u^*)'(x)-(u^*)'(y)}=\frac{x-y}{(u')^{-1}(x)-(u')^{-1}(y)}
\end{eqnarray}
is a free Stein kernel for $\mu$ with respect to the standard semicircular potential $\frac{1}{2}x^2$.
\end{theorem}
\begin{proof}
For a test function $f$, we have:
\begin{equation}
\int u'(x)f(x)d\nu_{u}(x)=\iint\frac{f(x)-f(y)}{x-y}d\nu_{u}(x)d\nu_{u}(y)\nonumber
\end{equation}
Take now $g$ a test function and set $f(x)=g(u'(x))$. Then the previous equation become:
\begin{eqnarray}
\int u'(x)g(u'(x))d\nu_{u}(x)&=&\iint\frac{g(u'(x))-g(u'(y))}{x-y}d\nu_{u}(x)d\nu_{u}(y)\nonumber\\
&=&\iint\frac{g(u'(x))-g(u'(y))}{u'(x)-u'(y)}\frac{u'(x)-u'(y)}{x-y}d\nu_{u}(x)d\nu_{u}(y)\nonumber
\end{eqnarray}
\bigbreak
Now put $\tilde{x}=(u^*)'(x)=(u')^{-1}(x)$ and $\tilde{y}=(u^*)'(y)=(u')^{-1}(y)$, and note that that this change of variables also sends $\mu$ to $\nu_{\mu}$ (and thus also for the product measure $\mu^{\otimes 2}$ and $\nu_u^{\otimes 2}$).Doing the same variable change in the left leg, we then get
\begin{equation}
\int xg(x)d{\mu}=\iint\frac{x-y}{(u^*)'(x)-(u^*)'(y)}\frac{g(x)-g(y)}{x-y}d\mu(x)d{\mu}(y).\nonumber
\end{equation}
And so we arrive at the desired conclusion.
\end{proof}
\qed
\bigbreak
\begin{flushleft}
We also remind to the reader that we denote the $L^2$ Kantorovitch-Wasserstein distance on the space of probability measures
\begin{equation}
        W_2(\mu,\nu)^2:=\inf \left\{\int \lvert x-y\rvert^2d\pi; \pi\in \mathscr{P}(\mathbb{R}^d\times \mathbb{R}^d),\pi(\cdot,\mathbb{R}^d)=\mu,\pi(\mathbb{R}^d,\cdot)=\nu\right\}.\nonumber
    \end{equation}
Now we recall that Cébron proved in \cite{C} that the quadratic Wasserstein distance between a compactly supported probability measure and the standard semicircular distribution is controlled by the free Stein discrepancy (a proposition which continues to hold in the multidimensional case, cf. Proposition 2.5 in \cite{C}, which can also be extended to any positive definite covariance matrix: Lemma 4.16 in \cite{Di}).
\end{flushleft}

\begin{theorem}(WS inequality, Cébron, proposition 2.7 in \cite{C})\label{wsc}
    Let $\mu$ a compactly supported probability measure, then:
    \begin{equation}
        W_2(\mu,\eta)^2\leq \Sigma^*(\mu|\nu_{\frac{x^2}{2}})^2
    \end{equation}
\end{theorem}
\begin{flushleft}
    In particular, this allows us to bound the free Stein discrepancy by an estimate of the non-commutative Hessian of $u$ in the non-commutative Sobolev space $H^{1}(\nu_u)$. More precisely, as we'll see in section \ref{sect4}, $u$ (if it is smooth enough) will necessarily be a solution of a variant of the free Monge-Ampère equation \cite{GS}, so the Stein discrepancy can be bounded by a suitable regularity estimate on a {\it free Kähler-Einstein}-type equation, as shown in the sequel.
\end{flushleft}

\begin{corollary}\label{cor1}
If $\mu$ satisfies the hypothesis of the previous theorem, then the quadratic Wasserstein distance between $\mu$ and the standard semicircular distribution $\mathcal{S}(0,1)$ is bounded as follows
\begin{eqnarray}
W_2(\mu,\eta)^2 &\leq &\iint\bigg[\frac{x-y}{(u^*)'(x)-(u^*)'(y)}-1\bigg]^2d\mu(x)d{\mu}(y)\nonumber\\
&= & \iint\bigg[\frac{u'(x)-u'(y)}{x-y}-1\bigg]^2d\nu_{u}(x)d{\nu}_{u}(y):=\lVert \mathscr{D}u-Id\rVert_{H^1(\nu_u)}^2\nonumber,
\end{eqnarray}
\end{corollary}
This means that if we are able to bound in $L^2$ the non-commutative Hessian $\mathscr{J}\mathscr{D}{u}$ averaging against $\nu_u^{\otimes 2}$, or equivalently estimates in $H^1(\nu_u)$, we are able to obtain estimates of transport distances to the semicircular law. 
\begin{flushleft}
In fact, following Fathi, Cébron and Mai \cite{FCM}, we denote the first-order non-commutative Sobolev space associated with a (non-commutative) distribution $\lambda$, (assumed to be diffuse to simplify exposure) as
\begin{equation}
    H^1(\lambda):=\left\{f:\mathbb{R}\mapsto \mathbb{R}, \iint\bigg(\frac{f(x)-f(y)}{x-y}\bigg)^2d\lambda(x)d\lambda(y)<\infty\right\},
\end{equation}
and its associated semi-norm
\begin{equation}
   \lVert f\rVert_{H^1(\lambda)}^2:=\iint\bigg(\frac{f(x)-f(y)}{x-y}\bigg)^2d\lambda(x)d\lambda(y).\nonumber
\end{equation}
\textbf{N.B:} To obtain a true Hilbert space, we recall that we have in fact to consider the quotient of $H^1(\lambda)$ by $N^1(\lambda)=\left\{(f,g)/ \mathscr{J}f=\mathscr{J}g\:, \lambda^{\otimes 2}-a.e\right\} $).
\bigbreak
In a shorthand, for $f,g \in H^1(\mu)$, we also denote
\begin{equation}
   \langle f,g\rangle_{H^1(\lambda)}=\iint\frac{f(x)-f(y)}{x-y}\frac{g(x)-g(y)}{x-y}d\lambda(x)d\lambda(y)=\iint \mathscr{J}f\cdot \mathscr{J}g\: d\lambda^{\otimes 2}\nonumber
\end{equation}
\end{flushleft}

However, at the moment we lack a free analogue of the important Klartag universal regularity estimates on convex solutions of the previous K\"ahler-Einstein PDE (based on earlier work by Kolesnikov \cite{KlaKol}), which in the classical case allows to obtain universal quantitative estimates (depending only on $p\geq 2$ and not on the target measure) on the second derivatives (in any direction) of the {\it moment map} $\varphi$ in the weighted Sobolev $L^p$-spaces $L^p(e^{-\varphi}dx), p\geq 2$ for the class of compactly supported log-concave measures $\mu$. Surprisingly, Klartag also proved a control on the third-order derivatives of $\varphi$, but only for now in $L^2$. He moreover proved a uniform bound $\lVert\Hess\varphi\rVert_{op}\leq c^{-1}$ under the uniform convexity assumption $\Hess V\geq c$.
\begin{proposition}(Klartag, \cite{Klart})
    Let $\mu$ be a log-concave probability measure supported on an open bounded convex set, and with a density bounded from above and below. Then the essentially continuous convex function $\varphi$ for which $\mu$ is the moment measure is $\mathcal{C}^2$ and satisfies for any $p\geq 1$ and $\theta \in \mathbb{S}^{d-1}$:
    \begin{equation}
        \int \lvert \langle \Hess \varphi(\nabla \varphi^*)\theta,\theta\rangle\lvert^p d\mu\leq 8^p\:p^{2p}\bigg(\int (x\cdot \theta)^2 d\mu\bigg)^p,
    \end{equation}
\end{proposition}

\section{A free toric \textrm{Kähler}-Einstein equation}\label{sect4}
\begin{flushleft}
The aim of this section is to show that the associated convex function $u$ of the variational problem considered by Bahr and Boschert \eqref {2.5} is in fact a solution of a variant of a free Monge-Ampère equation, and therefore to deduce that we can interpret the bound obtained for $W_2$ as a suitable regularity estimate on this free Monge-Ampère equation. 
\bigbreak
First we need to derive {\it a la Guionnet and Shlyakhtenko} a {\it free toric Kähler-Einstein} equation (Section 1.6 in \cite{GS}) in the smooth and strictly convex setting. Indeed, we now assume that the target measure $\mu$ is itself a free Gibbs measure associated with a smooth uniformly convex potential $u$ with $u''(x)>\kappa, \forall x \in \mathbb{R}$ for some $\kappa>0$, and so that $\mu:=\nu_u$, which we will assume to be centered, i.e. $\int xd\mu(x)=0$ (e.g. this condition is ensured if $u$ is even). We also recall that $\mu$ is supported on a compact interval $K=[a,b]$. In this case $\mu$ will itself be a free moment measure for some convex function $\varphi$ which we will assume for now to be smooth enough, which means at least of class $\mathcal{C}^4$ (we will explain the reason in detail in the next paragraphs).
\bigbreak
In the following, all improper integrals are taken in the sense of principal values, and we recall that for a function $\varphi$ which is $\mathcal{C}^2$ and strictly convex $\varphi''>0$, we have $\mathscr{J}\mathscr{D}\varphi>0$.
\bigbreak
Since $\mu:=\nu_u$ is a free Gibbs measure associated with a potential $u$ at least of class $\mathcal{C}^3$, then the following holds (see e.g. Ledoux, Popescu, theorem $1$ in \cite{LP}):
\begin{eqnarray}\label{39}
    2\int\frac{1}{x-y}d\mu(y)=u'(x),\:\:\mbox{$\forall x\in \supp(\mu)$},\nonumber
\end{eqnarray}
Then we denote by $\varphi$ the {\it free moment map} of $\mu$ (up to translations),
\newline
We then have by replacing $x$ and $y$ respectively by $\varphi'(x)$ and $\varphi'(y)$, since $\mu=(\varphi')_{\sharp}\nu_{\varphi}$:
\begin{equation}
    2\int \frac{1}{\varphi'(x)-\varphi'(y)}d\nu_{\varphi}(y)=u'(\varphi'(x)),\:\mbox{$\forall x\in \supp(\nu_{\varphi})$}\nonumber
\end{equation}
which can be rewritten as:
\begin{equation}
    2\int \frac{1}{x-y}\frac{1}{\mathscr{J}\mathscr{D}\varphi(x,y)}d\nu_{\varphi}(y)=u'(\varphi'(x)),\:\mbox{$\forall x\in \supp(\nu_{\varphi})$}\nonumber
\end{equation}
Multiplying by $\varphi''(x)$ on both side of this equation, we find:
\begin{equation}\label{kh}
    2\int \frac{1}{x-y}\frac{\varphi''(x)}{\mathscr{J}\mathscr{D}\varphi(x,y)}d\nu_{\varphi}(y)=\varphi{''}(x)\cdot u'(\varphi'(x)).
\end{equation}
Now subtracting on both side of \eqref{kh} by $\varphi'(x)$, using that $2\int \frac{1}{x-y}d\nu_{\varphi}(y)=\varphi'(x), \forall x \in\supp(\nu_{\varphi})$ in the left-hand side, we get
\begin{eqnarray}\label{KalEin}
    2\int \frac{1}{x-y}\frac{\varphi''(x)-\mathscr{J}\mathscr{D}\varphi(x,y)}{\mathscr{J}\mathscr{D}\varphi(x,y)}d\nu_{\varphi}(y)&=&\varphi''(x)\cdot u'(\varphi'(x))-\varphi'(x)\nonumber\\
    &=& [u(\varphi'(x))-\varphi(x)]',
\end{eqnarray}
Simple algebra shows that we have the following relations between \begin{equation}
\partial_x(\mathscr{J}\mathscr{D}\varphi(x,y))=\frac{1}{x-y}[\varphi''(x)-\mathscr{J}\mathscr{D}\varphi(x,y)]\nonumber
\end{equation}
And 
\begin{equation}\nonumber
    \int \frac{1}{x-y}\frac{\varphi''(x)-\mathscr{J}\mathscr{D}\varphi(x,y)}{\mathscr{J}\mathscr{D}\varphi(x,y)}d\nu_{\varphi}(y)=\partial_x\int \log(\mathscr{J}\mathscr{D}\varphi(x,y))d\nu_{\varphi}(y)\nonumber,
\end{equation}
Thus we reach, \begin{eqnarray}\label{KalEin2}
    2\partial_x\int \log(\mathscr{J}\mathscr{D}\varphi(x,y))d\nu_{\varphi}(y)&=&\varphi''(x)\cdot u'(\varphi'(x))-\varphi'(x)\nonumber\\
    &=& [u(\varphi'(x))-\varphi(x)]',
\end{eqnarray}
And finally by integrating \eqref{KalEin2}, we arrive to:
\begin{eqnarray}\label{4.3}
2\int\log(\mathscr{J}\mathscr{D}\varphi(x,y))d\nu_{\varphi}(y)=u(\varphi'(x))-\varphi(x)+\const.
\end{eqnarray}
The constant can be fixed by requiring that both sides of the equation vanish at $x=0$.
\bigbreak
This last equation looks very much like the {\it toric K\"ahler-Einstein} equation \eqref{mongA} in its logarithmic form, for an initial measure with density $e^{-\psi}$:
\begin{eqnarray}
    \log(\det \Hess \varphi)=\psi(\nabla \varphi)-\varphi.
\end{eqnarray}
\begin{flushleft}
    Now we can see that the fixing of $u=\frac{x^2}{2}$ in \eqref{4.3} (the constant disappears), so that $\mu: =\nu_{\frac{x^2}{2}}$ is the standard semicircular distribution, leads necessarily to $\varphi=\frac{x^2}{2}$, so that the unique fixed point of the map $\mu\mapsto \nu_{\varphi}$ is necessarily the quadratic potential $\varphi=\frac{x^2}{2}$, which justifies the interpretation given in the corollary \eqref{cor1} that the transport distances to the semicircular law are controlled by a suitable regularity estimate on the weak solution of a variant of the free Monge-Ampère equation. This result implies in particular the stability of the fixed point of the free Monge-Ampère equation.
    \bigbreak
    \begin{flushleft}
In the classical case, once the existence of such an optimal transport map was proved, the question of the regularity of the transport map was raised. Cafarelli, in several articles, e.g. \cite{Cafa2,Cafa}, gives regularity results about the transport map under various assumptions on the target and source measures and proves the well-known Caffarelli contraction theorem, which states that the Brenier map that sends the Gaussian measure to a uniformly log-concave measure is Lipschitz. Klartag in \cite{Klart} actually proved the moment analogue of Caffarelli's result, which we restate here.
    \begin{theorem}(Klartag, proposition 2 in \cite{Klart})
    The moment map $\varphi$ of a centered uniformly log-concave probability measure $\mu:=e^{-V}dx$, where $\Hess V\geq \epsilon Id$, for some $\epsilon>0$, has a uniformly bounded Hessian,
    \begin{equation}
        \lVert\Hess \varphi\rVert_{op}:=\underset{x\in \mathbb{R}^d}{\sup}\lVert\Hess \varphi(x)\rVert_{op}<\epsilon^{-1}. 
    \end{equation}
    \end{theorem}
        From this statement Fathi deduced that in this case the moment Stein kernel is uniformly bounded.
\begin{corollary}(Fathi, Corollary 2.4 in \cite{mm})
    Suppose $\mu$ is uniformly log-concave, i.e. $\mu=e^{-V}dx$, with $\Hess V\geq \epsilon Id$, for some $\epsilon>0$, then its moment Stein kernel
    $\Gamma=\Hess \varphi(\nabla\varphi^*)$ is uniformly bounded: $\rVert \Gamma\rVert_{op}\leq \epsilon^{-1}$.
\end{corollary}
\bigbreak
In fact, we can prove that this statement still holds in the free case when the target measure is itself a free Gibbs measure associated with a uniformly convex potential $V$, under an additional regularity assumption that we unfortunately could not remove (the $\mathcal{C}^4$-smoothness of the moment map, since a general regularity theory for free Monge-Ampère equations is still lacking), thus providing a first step towards a more general Free {\it Klartag contraction theorem}.
\begin{proposition}(Free Klartag contraction theorem)\label{cafarelli}
    Let $\nu_u$ be a centered free Gibbs measure with potential $u$ which is supposed $\mathcal{C}^2$ and which is such that for some $\epsilon>0$, $x\mapsto u(x)-\frac{\epsilon}{2} x^2$ is convex. If the (free) moment map $\varphi$ of $\nu_u$ is $\mathcal{C}^4$, then $\varphi'$ is $\epsilon^{-1}$-Lipschitz on $\mathbb{R}$, and thus $\lVert \mathscr{J}\mathscr{D}\varphi\rVert_{{\infty}}:=\underset{(x,y)\in \mathbb{R}^2}{\sup}\bigg\lvert \mathscr{J}\mathscr{D}\varphi(x,y)\bigg\rvert\leq \epsilon^{-1}$.
\end{proposition}
\begin{proof}
   We divide the proof into several steps, closely following the original proof of Cafarelli \cite{Cafa}, which is based on a maximum principle.
    \bigbreak
    First, note that the target measure has compact support, and that $\varphi$ is also the optimal transport map between $\nu_{\varphi}$ and $\nu_{u}$.
     We need to prove that if $\varphi$ denotes the free moment map, then $\lVert \mathscr{J}\mathscr{D}\varphi\rVert_{{\infty}}$ is bounded. 
     Now we assume that $\varphi$ is $\mathcal{C}^4$ (this hypothesis is already mentioned in the remark before $(7.9)$ in \cite{Hou}, in order to ensure the smoothness of an integral operator of the kind we are going to deal with). So if we can uniformly bound the second derivative $\varphi''$, the proof is complete. To do this, however, we must assume that the function $\varphi''$ reached is maximum at some point $x=x_*$. This is always the case because the target measure (and even the source measure) is compactly supported. This ensures that $\underset{x\rightarrow \pm \infty}{\lim}\varphi''(x)=0$, and so the positive maximum of $\varphi''$ is reached at some point $x=x_*$.
\begin{flushleft}
    Let us start with the equation \eqref{KalEin2}, which is the logarithmic (derivative) version of the free K\"ahler-Einstein equation.
    \begin{equation}
 2\partial_x\int \log(\mathscr{J}\mathscr{D}\varphi(x,y))d\nu_{\varphi}(y)=\varphi{''}(x)\cdot u'(\varphi'(x)).\nonumber
    \end{equation}
Taking the derivative of this equation, we get
\begin{eqnarray}\label{logmonge}
   2 \partial_{xx}\bigg(\int \log(\mathscr{J}\mathscr{D}\varphi(x,y))d\nu_{\varphi}(y)\bigg)=\varphi'''(x).u'(\varphi'(x))+\varphi''^{2}(x)u''(\varphi'(x))-\varphi''(x).
\end{eqnarray}
Now suppose that $\varphi''$ has a maximum at $x=x_*$, then $\varphi'''(x_*)=0$ (remark also that in this case $\varphi''''(x_*)\leq 0$, even if it's not needed in the proof). 
\bigbreak
Let us first work on the left hand side to prove that:
\begin{equation}
     \partial_{xx}\bigg(\int \log(\mathscr{J}\mathscr{D}\varphi(x,y))d\nu_{\varphi}(y)\bigg)\bigg|_{x=x_*}\leq 0.
\end{equation}
We have:
\begin{eqnarray}\label{maxim}
    \partial_{xx}\int \log(\mathscr{J}\mathscr{D}\varphi(x,y))d\nu_{\varphi}(y)&=&\partial_x\int \frac{1}{x-y}\frac{\varphi''(x)-\mathscr{J}\mathscr{D}\varphi(x,y)}{\mathscr{J}\mathscr{D}\varphi(x,y)}d\nu_{\varphi}(y)\\
    &=&-\int \frac{1}{(x-y)^2}\frac{\varphi''(x)-\mathscr{J}\mathscr{D}\varphi(x,y)}{\mathscr{J}\mathscr{D}\varphi(x,y)}d\nu_{\varphi}(y)\nonumber\\
   &+&\int \frac{1}{x-y}[\varphi'''(x)-\partial_x\mathscr{J}{\mathscr{D}}\varphi(x,y)\mathscr{J}\mathscr{D}\varphi(x,y)]d\nu_{\varphi}(y)\nonumber\\
   &-&\int \frac{1}{x-y}\partial_x\mathscr{J}{\mathscr{D}}\varphi(x,y)[\varphi''(x)-\mathscr{J}\mathscr{D}\varphi(x,y)]
    d\nu_{\varphi}(y).\nonumber
\end{eqnarray}
\newline
Now we know that $\varphi''$ is maximum at $x=x_*$ with thus $\varphi'''(x_*)=0$.
Then \eqref{maxim} at $x=x_*$ become:
\begin{align*}
     &\partial_{xx}\bigg(\int \log(\mathscr{J}\mathscr{D}\varphi(x,y))d\nu_{\varphi}(y)\bigg)\bigg|_{x=x_*}\nonumber\\
    &=-\int \frac{1}{(x_*-y)^2}\frac{\varphi''(x_*)-\mathscr{J}\mathscr{D}\varphi(x_*,y)}{\mathscr{J}\mathscr{D}\varphi(x_*,y)}d\nu_{\varphi}(y)
   -\int \frac{1}{x_*-y}\partial_x(\mathscr{J}{\mathscr{D}}\varphi(x,y))|_{x=x_*}\mathscr{J}\mathscr{D}\varphi(x_*,y)]d\nu_{\varphi}(y)\nonumber\\
   &-\int \frac{1}{x_*-y}\partial_x(\mathscr{J}{\mathscr{D}}\varphi(x,y))|_{x=x_*}[\varphi''(x_*)-\mathscr{J}\mathscr{D}\varphi(x_*,y)]
    d\nu_{\varphi}(y),\nonumber
\end{align*}
Now recalling the relation between mean value and derivatives:
\begin{equation}
\partial_x\mathscr{J}\mathscr{D}\varphi(x,y)=\frac{1}{x-y}[\varphi''(x)-\mathscr{J}\mathscr{D}\varphi(x,y)],
\end{equation}
the trivial bound,
\begin{equation}
    \mathscr{J}{\mathscr{D}}\varphi(x,y)\leq \varphi''(x_*),
\end{equation}
And using that $\forall y\in \mathbb{R}$, $\mathscr{J}\mathscr{D}\varphi(x_*,y)\geq 0$, we get
\begin{equation}
    \int \frac{1}{(x_*-y)^2}\frac{\varphi''(x_*)-\mathscr{J}\mathscr{D}\varphi(x_*,y)}{\mathscr{J}{\mathscr{D}}\varphi(x_*,y)}d\nu_{\varphi}(y)\geq 0,\nonumber
\end{equation}
Since we have that \begin{equation}
    \frac{1}{x_*-y}\partial_x(\mathscr{J}\mathscr{D}\varphi(x,y))|_{x=x_*}=\frac{1}{(x_*-y)^2}[\varphi''(x_*)-\mathscr{J}\mathscr{D}\varphi(x_*,y)]\geq 0,\nonumber\end{equation}
\bigbreak
It leads to,
\begin{equation}
    \int \frac{1}{x_*-y}\partial_x(\mathscr{J}{\mathscr{D}}\varphi(x,y))|_{x=x_*}\mathscr{J}\mathscr{D}\varphi(x_*,y)]d\nu_{\varphi}(y)\geq 0,\nonumber
\end{equation}
and 
\begin{equation}
    \int \frac{1}{x_*-y}\partial_x(\mathscr{J}{\mathscr{D}}\varphi(x,y))|_{x=x_*}[\varphi''(x_*)-\mathscr{J}\mathscr{D}\varphi(x_*,y)]
    d\nu_{\varphi}(y)\geq 0,\nonumber
\end{equation}
\bigbreak
Combining all these inequalities, we have therefore (because of the minus sign in front of each integrals)
\begin{equation}
     \partial_{xx}\bigg(\int \log(\mathscr{J}\mathscr{D}\varphi(x,y))d\nu_{\varphi}(y)\bigg)\bigg|_{x=x_*}\leq 0,\nonumber
\end{equation}
\bigbreak
Hence we have from \eqref{logmonge} that (recall that  $\varphi''$ is maximum at $x=x_*$, thus $\varphi'''(x_*)=0$)
\begin{equation}
         \underbrace{\varphi'''(x_*)}_{=0}.u'(\varphi'(x_*))+\varphi''^2(x_*).u''(\varphi''(x_*))-\varphi''(x_*)\leq 0,\nonumber
\end{equation}
\bigbreak
Now $\varphi''(x)\geq 0$ since $\varphi$ is convex, and we can assume that $\varphi''(x_*)>0$ since otherwise there's nothing to prove (since in that case it would lead to $\varphi''\equiv 0$), and by assumption $u''(x)\geq \epsilon$ since $x\mapsto u(x)-\frac{\epsilon}{2}x^2$ is convex. From this we obtain,
\begin{equation}
    \varphi''(x_*)[\varphi''(x_*).u''(\varphi'(x_*))-1]\leq 0,\nonumber
\end{equation}
This is equivalent to 
\begin{equation}
    \varphi''(x_*)\leq \frac{1}{u''(\varphi'(x_*))},\nonumber
\end{equation}
Which finally gives that:
\begin{equation}
    \varphi''(x_*)\leq \epsilon^{-1}.
\end{equation}
and concludes the proof.
\end{flushleft}
\end{proof}
\qed
        
    \end{flushleft}
\begin{corollary}\label{cor3}
    Let $\nu_u$ be a centered free Gibbs measure with potential $u$ which is $\mathcal{C}^2$ and which is such that for some $\epsilon>0$, $x\mapsto u(x)-\frac{\epsilon}{2} x^2$ is convex. If the (free) moment map $\varphi$ is $\mathcal{C}^3$, then its free moment Stein kernel is uniformly bounded, 
    $\lVert \mathscr{J}\mathscr{D}u(\mathscr{D}\varphi^*)\rVert_{\infty}:=\underset{(x,y)\in \mathbb{R}^2}{\sup}\bigg\lvert \mathscr{J}\mathscr{D}u(\mathscr{D}\varphi^*(x),\mathscr{D}^*\varphi(y))\bigg\rvert\leq \epsilon^{-1}$.
\end{corollary}
\begin{proof}
    We know that $A(x,y):=\mathscr{J}\mathscr{D}\varphi(\varphi^*(x),\varphi^*(y))$ is a free Stein kernel for $\nu_u$ with respect to the standard semicircular potential.
    \bigbreak
    From the previous theorem \ref{cafarelli}, we know that for all $\underset{x \in \mathbb{R}}{\sup}\:\lvert \varphi''(x)\rvert\leq \epsilon^{-1}$.
    Therefore $\lVert \mathscr{J}\mathscr{D}\varphi\rVert_{\infty}\leq \epsilon^{-1}$, and the conclusion follows.
    
\end{proof}
\qed
\begin{remark}
    Following the same scheme (inspired by the breakthough machinery invented by Calabi \cite{Calabi} to prove regularity results for the Monge–Ampère equation), it is also not difficult to show by taking second derivatives in the logarithmic free Monge-Ampère equation that, assuming $u,\varphi$ to be $\mathcal{C}^5$-smooth, we must have the maximum of $\varphi''$ reached at point $x=x_*$:
    \begin{equation}
        \partial_{xxx}\bigg(\int \log(\mathscr{J}\mathscr{D}\varphi(x,y))d\nu_{\varphi}(y)\bigg)\bigg|_{x=x_*}\leq 0.
    \end{equation}
which implies some relations giving a third-order regularity bound, as Klartag and Kolesnikov \cite{Klart2} have done in the classical case. As we haven't found any applications for the free Stein's method yet, we propose to present this point of view in another paper.
\newline
It would indeed be very interesting to remove the $\mathcal{C}^4$ assumption on $\varphi$ in the previous theorems. It seems (as mentioned in another context in \cite{Hou}, theorem $11$) that only $\mathcal{C}^3$-smoothness is needed, but we leave this question for further investigations. 
\end{remark}
\begin{flushleft}
As a straightforward corollary (following identically Fathi's proof of Theorem $3.4$ in \cite{mm}, which is straightforwardly adapted to the free case), we obtain the following rate of convergence in the free CLT for the free Gibbs measure associated with a $\epsilon$-uniformly convex potential, which explicitly depends on the convexity constant.
\begin{corollary}
   Suppose $\mu_n$ is the law of $\frac{1}{\sqrt{n}}\sum_{i=1}^n X_i$, where $(X_i)_{i\in \mathbb{N}^*}$ is a sequence of normalised, freely independent and i.i.d bounded random variables with the law $\nu_u$, a centered free Gibbs measure associated with $u$, a $\epsilon$-uniformly convex potential satisfying the assumptions of the corollary \ref{cor3}. Then there exists $C>0$ (universal) such that we have 
    \begin{eqnarray}
        W_2(\mu_n,\eta)^2&\leq& \frac{C}{\epsilon n}
    \end{eqnarray}
\end{corollary}
\end{flushleft}

\bigbreak
More generally, if we want to transport the semicircular law via a smooth map $T$ to some free Gibbs measure associated with $u:\mathbb{R}\rightarrow \mathbb{R}$, a $\epsilon$-uniformly convex potential, then it is not difficult to show that in this case $T$ must solve the following free Monge-Ampère equation.
\begin{equation}\label{monge-free}
    2\int\log(\mathscr{J}T(x,y))d\eta(y)=u(T'(x))-T(x)+\const,\: \forall x \in [-2,2].
\end{equation}
We can then derive the following variant of the contraction theorem. It can be obtained mutadis mutandis as for the moment maps version of the theorem \ref{cafarelli}.
\begin{theorem}(Free Cafarelli contraction theorem) \label{th7}
Let $\eta$ be the standard semicircular distribution, and $\nu_u$ be a free Gibbs measure associated with a smooth $u:\mathbb{R}\rightarrow \mathbb{R}$ be a $\epsilon$-uniformly convex potential.
Let $T$ be the optimal transport map sending $\eta$ to $\nu_u$. If we assume that $T$ is $\mathcal{C}^3$, then $T$ is $\sqrt{\epsilon^{-1}}$-Lipschitz.
\end{theorem}

\bigbreak
This allows us to derive the following stability result, which is the free counterpart of the Bakry-\'Emery stability theorem (see \cite{FathCour}), giving a characterisation of the semicircular distribution among the class of centered and isotropic free Gibbs measures associated with a $1$-uniformly convex potential. We also suggest that this statement may hold even under the weaker transport regularity hypothesis.
\newline
We claim no originality in the following proof, which is essentially identical to Proposition 2 of Courtade and Fathi \cite{FathCour}. 
\begin{corollary}\label{cafafree}
    Let $\mu:=\nu_u$ be a free Gibbs measure supported on $[-2,2]$ with $u''(x)\geq 1,\forall x\in \mathbb{R}$ which is supposed to be centered and of variance $1$. If the transport map $T$ which sends the semicircular law $\eta$ to $\mu$ is $\mathcal{C}^3$, then $\mu=\eta$.
\end{corollary}
\begin{proof}
    First, note that in this case both measures have compact support and thus moments of order $2$, and are also absolutely continuous with respect to the Lebesgue measure (a fortiori non-atomic), so the Brenier theorem states that there exists such an optimal transport for quadratic cost (and is also the gradient of a convex function). Let $T$ be the transport map. Since $T$ is by assumption $\mathcal{C}^3$, we deduce by \ref{th7} that $T$ is $1$-Lipschitz in this case. Then we compute
    \begin{eqnarray}
        2&=&\int \lvert x-y\rvert^2d\mu(x)d\mu(y)\nonumber\\
        &=&\int \lvert T(x)-T(y)\rvert^2d\eta(x)d\eta(y)\nonumber\\
        &\leq& \int \lvert x-y\rvert^2d\eta(x)d\eta(y)\leq 2.
    \end{eqnarray}
    Therefore, we must have $\lvert T(x)-T(y)\rvert=\lvert x-y\rvert, \: \eta^{\otimes 2}-a.s\textbf{}$. Since $T$ is continuous (and $\eta$ is diffuse), this equality holds almost everywhere on $[-2,2]$. So we can conclude that $T$ is an affine map on $[-2,2]$. Since $\mu$ (which is supported exactly on $[-2,2]$) is centered and of variance one, $T$ must be the identity map.
\end{proof}
\qed
\begin{remark}
    It is interesting to note that we cannot generalise our theorem \ref{cafafree} if the support of the free Gibbs measure is arbitrary, let's say a compact interval $[a,b]$ other than $[-2,2]$. In fact, the main argument of the classic case of Courtade and Fathi (Proposition $2$ in \cite{FathCour}) is as follows: first, they prove that the result holds if the measure has full support, i.e. $\supp(\mu)=\mathbb{R}^d$. Second, if it doesn't (i.e. if the support is arbitrary), they prove that the conclusions remain true by using a Gaussian convolution and a proper renormalisation. In our case, using the same scheme, we take the free convolution of $\mu$ with a standard semicircular $\eta$ and then rescale it so that the new measure $\nu$ is still isotropic. In fact, this doesn't change the $1$-convexity hypothesis of the potential, nor its smoothness, since it is still $\mathcal{C}^2$. However, the whole conclusion collapses due to the failure of \textit{Cramer's theorem} in the free setting, which states the following.
    \end{remark}
\begin{theorem}Failure of Cram\'er's theorem (Voiculescu, Bercovici, \cite{BercV}). 
\newline
    There exists $\mu,\nu\in \mathscr{P}(\mathbb{R})$ compactly supported probability measure on the real line which are not semicircular and such that their free convolution $ \mu\boxplus\nu$ is semicircular.
\end{theorem}
However, we think that the result might be true thanks to a possible weak version of the free Cramer's theorem, but unfortunately we have not been able to prove it, and not even the first proposition in this direction, which would be a free version of Prekopa's theorem, asserting in the classical case that log-concavity is preserved (or strongly preserved) by the usual convolution.
\begin{conjecture}(Weak free Cramer's conjecture)
   Let $\mu:=\nu_{\phi}$ and $\nu:=\nu_{\psi}$ be two free Gibbs measures associated with two strictly $\mathcal{C}^2$ convex potentials $\phi,\psi:\mathbb{R}\rightarrow \mathbb{R}$. If the free convolution $\mu\boxplus\nu$ is semicircular, then both $\mu$ and $\nu$ are semicircular.
\end{conjecture}

 \begin{flushleft}
 We now discuss the idea and heuristics about the (classical and universal) regularity estimates obtained on the solution $\varphi$ of the {\it toric Kähler-Einstein} equation \eqref{mongB} for an initial log-concave measure (not necessarily strictly) supported on a bounded convex set $K$ which are due to Klartag \cite{Klart}, based on an idea of Kolesnikov \cite{Kol} (see also an alternative proof using the Bakry-Émery calculus proposed by Fathi: c.f. Theorem $3.2$ in \cite{mm}). Indeed, let's consider the Hessian metric given by the Riemannian metric tensor $g=\Hess\varphi$ on $\mathbb{R}^d$, then we can show that if the initial measure $\mu$ is log-concave, the metric measure space $M_{\varphi}^*:=(\mathbb{R}^d, g,e^{-\varphi}dx$) has Ricci curvature bounded from below by $1/2$, and one can check that the derivative in any direction of the {\it moment map} $\varphi$ is an eigenfunction always associated with the eigenvalue $-1$ of the Laplacian on $M$ (and containing all linear functional). Finally, geometric arguments (estimates in $L^p$ of the gradient of an eigenfunction of the Laplacian on a manifold with positive curvature) or simple calculations using the Bakry-Emery calculus (reformulated in this language by Fathi) suffice to conclude. In a more probabilistic way, this means that pushforward of the (classical) diffusion $(X_t)_{t\geq 0}$ (see equation \eqref{class}) via $\nabla \varphi^{*}$ is in fact the (standard) Brownian motion on the dual weighted Riemmanian space $M:= (K,(\Hess \varphi)^{-1},\mu)$ (where $K=\supp(\mu)$) has $\mu$ as invariant measure provided that it satisfies the conditions of theorem \eqref{3.3}. This Riemannian space $M$ is in fact \textit{stochastically complete} (i.e. this canonical Brownian motion of this manifold is non-explosive and never reaches the boundary of $K$). The curvature condition then "heuristically" ensures that the behaviour of the Brownian motion on this manifold cannot be worse than the behaviour of the Brownian motion on the sphere.
\end{flushleft}
\begin{flushleft}
In the sequel we continue our exposure with the notations used above.
\newline
        Following the diffusion point of view of Kolesnikov \cite{Kol} section 2, we can observe that a part of the above statements still hold in the free context (thus providing the first step towards proving the free Klartag regularity estimates, but not the main conclusion).
        In fact, if we reformulate the above calculations a bit and introduce the following:
        \begin{definition}
        Let $K_{\varphi}:=\supp(\nu_{\varphi})$. We define the following free diffusion operator, which we call the Laplacian on the non-commutative metric measure space $M_{n.c}^*:=(K_{\varphi},\mathscr{J}\mathscr{D}\varphi, \nu_{\varphi})$ and which is defined for functions $f$ of class $\mathcal{C}^2$ as, 
        \begin{eqnarray}\label{lphi}
            \mathscr{L}_{\varphi}f(x)
            &:=&2\int \frac{1}{x-y}\frac{f'(x)-\mathscr{J}f(x,y)}{\mathscr{J}\mathscr{D}\varphi(x,y)}d\nu_{\varphi}(y)-f'(x)\cdot u'(\varphi'(x))\nonumber\\
            &=&2\int \frac{f'(x)(x-y)-(f(x)-f(y))}{(x-y)^2}(\mathscr{J}\mathscr{D}\varphi(x,y))^{-1}d\nu_{\varphi}(y)-f'(x)\cdot u'(\varphi'(x)),\nonumber
        \end{eqnarray}
        \end{definition}
\begin{remark}
    More generally a non-commutative Hessian manifold of the form $M_{n.c}^*:=(K_{\varphi},\mathscr{J}\mathscr{D}\varphi, \nu_{\varphi})$ is said to be isomorphic to its dual as $M_{n.c}^*:=(K_{\mu},\mathscr{J}\mathscr{D}\varphi^*, \mu)$, the diffeomorphism from $M_{n.c}^*$ to $M_{n.c}$ the diffeomorphism map being given by $x\mapsto \varphi'(x)$.
\end{remark}

    Therefore, from the previous calculations we can observe that the non-commutative Laplacian associated with $M_{n.c}$ has an interesting spectrum: The first non-zero eigenvalue is $-1$, and the corresponding eigenspace contains all linear functions. This can be seen in particular from equation \eqref{KalEin}. We can also notice the dependence w.r.t the free Riemannian metric (weight) $\mathscr{J}\mathscr{D}\varphi$ or \eqref{KalEin2}. In fact, we will see that the right multidimensional variant of this Laplacian will be given by a equation which look like \eqref{KalEin2} from which we will deduce the same kind of results. 
    \begin{proposition}
    The function $\varphi'$ is an eigenfunction of the operator $\mathscr{L}_{\varphi}$, always associated with the eigenvalue $-1$:
  \begin{equation}
      \mathscr{L}_{\varphi}\varphi'=-\varphi',\nonumber
  \end{equation}
  And to denote the {\it carré du champ} operator
  \begin{equation}
      \Gamma_{\varphi}(f,g)=(\mathscr{J}\mathscr{D}\varphi)^{-1}\cdot \mathscr{J}f\cdot \mathscr{J}g\nonumber
  \end{equation}
  and in a shorthand $\Gamma_{\varphi}(f):=\Gamma_{\varphi}(f,f)$, we can check that:
  \begin{equation}
      \Gamma_{\varphi}(\mathscr{D}\varphi)=\mathscr{J}\mathscr{D}\varphi,
  \end{equation}
  \end{proposition}
We can be more precise and show that this Laplacian does in fact generate a non-commutative Dirichlet form.
  \begin{lemma}
  For $f,g\in \mathcal{C}_{0}^{\infty}(\mathbb{R})$, we have:
  \begin{equation}
      \iint (\mathscr{J}\mathscr{D}\varphi)^{-1}\cdot \mathscr{J}f\cdot \mathscr{J}g\:d\nu_{\varphi}^{\otimes 2}=-\int f\cdot \mathscr{L}_{\varphi}gd\nu_{\varphi}=-\int g\cdot \mathscr{L}_{\varphi}fd\nu_{\varphi}
  \end{equation}
  \begin{proof}
      Denoting in a shorthand $\psi=\varphi^*$, we have therefore that $\mathscr{J}\mathscr{D}\psi=(\mathscr{J}\mathscr{D}\varphi(\mathscr{D}\psi))^{-1}$.
      \bigbreak
      Now notice that $\mu$ is pushed forward by $\psi'$ to $\nu_{\varphi}$. So we can check that the following for almost all $x\in \mathbb{R}$:
      \begin{align}
          &\mathscr{L}_{\varphi}f 
          (\psi'(x))\nonumber\\
          &=2\int \frac{f'(\psi'(x))(\psi'(x)-\psi'(y))-(f(\psi'(x))-f(\psi'(y)))}{(\psi'(x)-\psi'(y))^2}\mathscr{J}\mathscr{D}\psi(x,y)d\mu(y)-f'(\psi'(x))\cdot u'(x)\nonumber\\
          &=\mathscr{J}^*_u(\mathscr{J}f(\mathscr{D}\psi))(x),\nonumber
      \end{align}
      Hence, it is the divergence of the non-commutative vector field $\mathscr{J}f(\mathscr{D}\psi)$ with respect to $\mu:=\nu_u$, where $\mathscr{J}^*_u$ denote the adjoint of $\mathscr{J}$ with respect to the inner product $L^2(\nu_u)$, 
    that is for test functions $p,q,r$ we have:
      \begin{equation}
          \langle \mathscr{J}p,q\otimes r\rangle_{L^2(\nu_u^{\otimes 2})}=\langle p,\mathscr{J}^*_u(q\otimes r)\rangle_{L^2(\nu_u)},\nonumber
      \end{equation}
      
      \bigbreak
      Indeed, let's denote $M_{u}f:=-2\partial_x(I\otimes \nu_{u})\mathscr{J}f+u'\mathscr{D}f$ as the infinitesimal generator of the free Langevin diffusion:
    \begin{equation}
    dX_t=-\mathscr{D}u(X_t)+\sqrt{2}dS_t,\nonumber
    \end{equation} 
    Then, by using the remark right after $(6.9)$ in \cite{Hou} (which is based on \cite{V} Corollary 4.4 and Proposition 3.5), we have $M_u=\mathscr{J}_u^*\mathscr{J}$, on the set of functions $f$ which are at least of class $\mathcal{C}^2$, and moreover, 
    \begin{equation}
        \langle M_{u}f,f\rangle_{\mu}=\iint\bigg(\frac{f(x)-f(y)}{x-y}\bigg)^2d\mu(x)d\mu(y),\nonumber
    \end{equation}
      From which we finally deduce that (where we used the chain rule for the non-commutative derivative \eqref{chain} in the third equality):
      \begin{eqnarray}
          \int f\cdot \mathscr{L}_{\varphi}gd\nu_{\varphi}&=&\int f(\psi')\cdot\mathscr{L}_{\varphi}g(\psi')\:d\mu
          =-\iint \mathscr{J}\mathscr{D}\psi\cdot \mathscr{J}f(\psi')\cdot \mathscr{J}g(\psi')\:d{\mu}^{\otimes 2}\nonumber\\
          &=&-\iint (\mathscr{J}\mathscr{D}\varphi)^{-1}\cdot \mathscr{J}f\cdot \mathscr{J}g\:d{\nu}_{\varphi}^{\otimes 2},\nonumber
      \end{eqnarray}
  \end{proof}
  \qed
  \end{lemma}
  Put another way, it implies that the function $\varphi$ is a solution of a quasilinear free diffusion equation, where $\mathscr{L}_{\varphi}$ is the generator of the non-commutative Dirichlet form.
  \begin{equation}
      \mathcal{E}(f):=\iint (\mathscr{J}\mathscr{D}\varphi)^{-1}\cdot (\mathscr{J}f)^2\:d\nu_{\varphi}^{\otimes 2},
  \end{equation}
We can also deduce that if we denote $(P_t)_{t\geq 0}$ as the semigroup acting on Lipschitz functions induced by $\mathscr{L}_{\varphi}$, then for all $t\geq 0$ we have $P_t\varphi'=e^{-t}\varphi' $ and $\Gamma_{\varphi}(P_t\varphi')=e^{-2t}\mathscr{J}\mathscr{D}\varphi $. 
    \end{flushleft}
\end{flushleft}
\end{flushleft}
\bigbreak
\begin{flushleft}
    In the following, we will assume that we are given a tracial $W^*$-probability space $(\mathcal{M},\tau)$, which is supposed to contain a copy of $L(\mathbb{F}_{\infty})$ and to be filtered:
\newline
This means that there exists a free Brownian motion $(S_t)_{t\geq 0}$ and its corresponding filtration $(\mathcal{A}_t)_{t\geq 0}$ in $(\mathcal{M},\tau)$. 
\end{flushleft}
\begin{flushleft}
We then follow the "diffusion" approach of Klartag, section $4$ in \cite{Klart}, which uses diffusion processes and, in particular, the notion of "{\it stochastically complete}" weighted Riemannian manifolds to study the transport of measures. This approach will allow us to formally derive the apparently correct "bimodular" \textit{carré-du-champ} operators associated with this free diffusion. However, we will not investigate the global well-posedness of such free SDE's, leaving this for further investigations. We also mention that we insist on the fact that in the free case it seems that two \textit{carre-du-champ operator} appear (one is \textit{uni-modular} and was defined by Guionnet and Shlyakhtenko, section 6 in \cite{bak}). To the best of our knowledge, an important open problem in the free case is to prove or disprove that free functional inequalities (LSI, TI, HWI \cite{ov}, HSI \cite{LNP}, etc...) can be obtained using a non-commutative Bakry-Émery criterion \footnote{This seems to be the first time that the "bimodular" version of the \textit{carré-du-champ} that we present appears explicitly, although it was probably well known to the specialist community.}. Note that the approach we present here was largely inspired by the formulas found by Houdré and Popescu: section 5 in \cite{Hou}, which basically boils down to the free Ornstein-Uhlenbeck case that we will present in detail immediately afterwards. 
\end{flushleft}
\begin{definition}
    We say that the non-commutative Hessian manifold 
    $M_{n. c}^*: =(\mathbb{R},\mathscr{J}\mathscr{D}\varphi, \nu_{\varphi})$ is "stochastically complete" in the free sense if the free Ito diffusion process with generator $\mathscr{L}_{\varphi}$ is defined for all times $t\in[0, +\infty)$, which amounts to proving that for any bounded initial data $X_0=z\in \mathcal{M}_{s. a}$ freely independent of the free Brownian motion $S$, the following free SDE holds:
    \begin{equation}\label{car}
dX_t=-u'(\varphi'(X_t))dt+(2(\mathscr{J}\mathscr{D}\varphi(X_t))^{-1})^{\frac{1}{2}}\sharp dS_t,
    \end{equation}
  does not explode in finite time. Note that here from the convexity assumption on $\varphi$, we have $\forall (x,y)\in \mathbb{R}^2,  \mathscr{J}\mathscr{D}u(x,y)\geq 0$. Hence, we see here the operator  $\mathscr{J}\mathscr{D}\varphi:L^2(\mathcal{M})\rightarrow L^2_+(\mathcal{M}\bar{\otimes}\mathcal{M})$, i.e. $\forall x\in L^2(M), \mathscr{J}\mathscr{D}\varphi(x)=Q^*Q,\:\mbox{for some Q} \in L^2(\mathcal{M}\bar{\otimes}\mathcal{M})$.
\end{definition}
Below we continue our exposition under the previous assumptions. 
\bigbreak
We also extend the action of the operator $\mathscr{L}_{\varphi}$ to the tensor product of functions as follows.
First, for tensor products of simple functions, we set
\begin{equation}
    \mathscr{L}_{\varphi}^{\otimes 2}=\mathscr{L}_{\varphi}\otimes id+id\otimes \mathscr{L}_{\varphi},
\end{equation}
and we extend its action by linearity.
\newline
We can then check that this operator is negative and closable, which allows us to extend its action to a larger domain.
\begin{definition}\label{def7}
    We define the second carré du champ $\Gamma_{\varphi,2}$ adapted to free diffusion \eqref{car} for smooth functions $f,g$ as
    \begin{equation}
        \Gamma_{\varphi,2}(f,g)=\frac{1}{2}\bigg(\Gamma_{\varphi}(\mathscr{L}_{\varphi}^{\otimes 2}(\Gamma_{\varphi}(f,g))-\Gamma(\mathscr{L}_{\varphi}f,g)-\Gamma_{\varphi}(f,\mathscr{L}_{\varphi}g)\bigg),
    \end{equation}
    and we set $\Gamma_{\varphi,2}(f):=\Gamma_{\varphi,2}(f,f)$.
\end{definition}
\begin{example}
    Let us first make a few remarks about the most fundamental case. Let's choose $u=\frac{x^2}{2}$, for which we know that $\varphi=\frac{x^2}{2}$. Thus the free diffusion $(X_t)_{t\geq 0}$ defined in \eqref{car} reduces to the free Ornstein-Uhlenbeck process $dX_t=-X_tdt+\sqrt{2}dS_t$, with infinitesimal generator $\mathscr{L}$ (the free Ornstein-Uhlenbeck operator), which has a pure point spectrum: $sp(-L)=\mathbb{N}$.
    \end{example}
    \begin{flushleft}
        In this case we get (since we fix $\varphi=\frac{x^2}{2}$, we denote this for convenience $\Gamma$ instead of $\Gamma_{\varphi}$),
    \end{flushleft}
    \begin{equation}
        \Gamma(f)=(\mathscr{J}f)^2,\nonumber
    \end{equation}
   We also recall the non-commutative Bochner formula $\mathscr{J}\mathscr{L}=(\mathscr{L}^{\otimes 2}-\Id)\mathscr{J}$, where $Id$ is the identity operator on $L^2([-2,2]^2,\eta^{\otimes 2})\simeq L^2([-2,2],\eta)\bar{\otimes} L^2([-2,2],\eta)$ (for such a proof, see point $(5.3)$ of Proposition $7$ in \cite{Hou}, note also the sign conventions, since they consider the number operator $N:=-\mathscr{L}$ instead), we first get
 \begin{eqnarray}
    \Gamma(f,\mathscr{L}f)=\Gamma(\mathscr{L}f,f)&=& (\mathscr{J}\mathscr{L}f)\cdot \mathscr{J}f\nonumber\\
    &=& (\mathscr{L}^{\otimes 2}(\mathscr{J}f)-\mathscr{J}f)\cdot \mathscr{J}f\nonumber\\
    &=& \mathscr{L}^{\otimes 2}(\mathscr{J}f)\cdot \mathscr{J}f-(\mathscr{J}f)^2.
  \end{eqnarray}
        \begin{flushleft}
    We then have (after some simple calculations) the following expression for $\Gamma_2$:
    \end{flushleft}
    \begin{eqnarray}
        \Gamma_2(f)&=&
        \frac{1}{2}\bigg[\mathscr{L}^{\otimes 2}((\mathscr{J}f)^2)-2\mathscr{L}^{\otimes 2}(\mathscr{J}f)\cdot \mathscr{J}f+2(\mathscr{J}f)^2\bigg],\nonumber
    \end{eqnarray}
    And it suffices to prove that:
    \begin{equation}
        \mathscr{L}^{\otimes 2}((\mathscr{J}f)^2)-2\mathscr{L}^{\otimes 2}(\mathscr{J}f)\cdot \mathscr{J}f\geq 0.
    \end{equation}
    We use a spectral approach to prove that a non-commutative Bakry-\'Emery criterion holds for this free diffusion:
    \bigbreak
    In fact, take $f(x)=U_n(x),\: n\geq 0$, where $(U_n)_{n\geq 0}$ is the sequence of Tchebychev polynomials of the second kind, defined recursively by setting $U_0=1$, $U_1=X$ and $XU_n=U_{n+1}-U_{n-1}$. It is known that they form an orthonormal basis of eigenfunctions of the free Ornstein-Uhlenbeck operator, and hence of $L^2(\eta)$, with $\mathscr{L}U_n=-nU_n$. Thus the \textit{chaotic decomposition} takes the following form,
    \begin{equation}
        L^2(\eta)=\bigoplus_{n=0}^{\infty}Ker(\mathscr{L}+nId),
    \end{equation}
    The non-commutative derivatives of the Tchebychev polynomials also have a nice form (see e.g. prop 5.3.9 in \cite{BS}, which also gives its "infinite dimensional" version when the non-commutative derivative is replaced by the free Malliavin gradient).
    \begin{equation}
        \mathscr{J}U_n=\sum_{k=1}^nU_{k-1}\otimes U_{n-k},\: \forall n\geq 1\nonumber
    \end{equation}
     Then, using the product formula for Tchebychev polynomials (which holds in full generality for Wigner-Ito integrals, see Proposition $5.3$ in Biane and Speicher \cite{BS}), we get
    \begin{equation}
        (\mathscr{J}U_n)^2=\sum_{j,k=0}^n\sum_{p=0}^{k\wedge j-1}\sum_{q=0}^{n-(k\vee j)} U_{k+j-2p-2}\otimes U_{2n-k-j-2q},\nonumber
    \end{equation}
And so we use $\mathscr{L}^{\otimes 2}$,
\begin{equation}
    \mathscr{L}^{\otimes 2}((\mathscr{J}U_n)^2)=\sum_{j,k=1}^n\sum_{p=0}^{k\wedge j-1}\sum_{q=0}^{n-(k\vee j)} (2p+2q-2n+2).U_{k+j-2p-2}\otimes U_{2n-k-j-2q},\nonumber
\end{equation}
We can also simply 
\begin{eqnarray}
    \mathscr{L}^{\otimes 2}(\mathscr{J}U_n)&=&(1-n)\sum_{k=0}^n U_{k-1}\otimes U_{n-k}\nonumber\\
    &=&(1-n)\mathscr{J}U_n,\nonumber
    \end{eqnarray}
And so we have,
\begin{eqnarray}
   \mathscr{L}^{\otimes 2}(\mathscr{J}U_n)\cdot \mathscr{J}U_n=(1-n)(\mathscr{J}U_n)^2 &=&\sum_{j,k=1}^n\sum_{p=0}^{k\wedge j-1}\sum_{q=0}^{n-(k\vee j)} (1-n).U_{k+j-2p-2}\otimes U_{2n-k-j-2q},\nonumber
\end{eqnarray}
This brings us to the \begin{equation}
    \mathscr{L}^{\otimes 2}((\mathscr{J}U_n)^2)-2\mathscr{L}^{\otimes 2}(\mathscr{J}U_n)\cdot \mathscr{J}U_n= 2\sum_{j,k=1}^n\sum_{p=1}^{k\wedge j-1}\sum_{q=1}^{n-(k\vee j)} (p+q). U_{k+j-2p-2}\otimes U_{2n-k-j-2q}\nonumber
\end{equation}
which, by tedious recursion, can be shown to be always positive.
\begin{flushleft}
So, by expanding $f$ in Tchebychev polynomials, this finally gives that the non-commutative Bakry-\'Emery criterion is satisfied:
\begin{equation}
    \Gamma_2(f)\geq \Gamma(f).
\end{equation}
\end{flushleft}
\begin{flushleft}
More generally, we can prove the following decomposition, but proving a general non-commutative Bakry-\'Emery criterion seems to be another challenge, since a spectral approach is impossible. Indeed, we recall from the work of Anshelevich \cite{anshe} that the only free diffusion operator of the form $\mathcal{L}_{\varphi}$ which has orthogonal polynomial eigenfunctions is the free Ornstein-Uhlenbeck operator, i.e. when $\varphi=\frac{x^2}{2}$. This is currently the main obstacle to proving a non-commutative curvature criterion.
\end{flushleft}
\begin{proposition}
    Let $\varphi$ satisfy the previous hypothesis, then for smooth test functions $f$ we have
    \begin{equation}
        \Gamma_{\varphi,2}(f)=\frac{1}{2}\bigg(\mathscr{L}_{\varphi}^{\otimes 2}((\mathscr{J}\mathscr{D}\varphi)^{-1}\cdot (\mathscr{J}f)^2)-2(\mathscr{J}\mathscr{D}\varphi)^{-1} \cdot \mathscr{L}_{\varphi}^{\otimes 2}(\mathscr{J}f)\cdot \mathscr{J}f +2(\mathscr{J}f)^2\bigg)
    \end{equation}
\end{proposition}
\begin{proof}
    \begin{equation}
        \Gamma_{\varphi}(f)=(\mathscr{J}\mathscr{D}\varphi)^{-1}\cdot (\mathscr{J}f)^2.\nonumber
    \end{equation}
 Then by using the non-commutative Bochner formula:
 \begin{equation}
     \mathscr{J}\mathscr{L}_{\varphi}=(\mathscr{L}_{\varphi}^{\otimes 2}-\mathscr{J}\mathscr{D}\varphi)\mathscr{J},\nonumber
 \end{equation}
 i.e for smooth test functions: $\mathscr{J}(\mathscr{L}_{\varphi}f)=\mathscr{L}_{\varphi}^{\otimes 2}(\mathscr{J}f)-\mathscr{J}\mathscr{D}\varphi\cdot\mathscr{J}f$,
 \begin{flushleft}
 We arrive to,
 \end{flushleft}
 \begin{eqnarray}
        \Gamma_{\varphi}(\mathscr{L}_{\varphi}f,f)&=&(\mathscr{J}\mathscr{D}\varphi)^{-1}\cdot\mathscr{J}\mathscr{L}_{\varphi}f\cdot\mathscr{J}f\nonumber\\
        &=&(\mathscr{J}\mathscr{D}\varphi)^{-1} (\mathscr{L}_{\varphi}^{\otimes 2}(\mathscr{J}f)- \mathscr{J}\mathscr{D}\varphi\cdot \mathscr{J}f)\cdot \mathscr{J}f\nonumber\\
        &=&(\mathscr{J}\mathscr{D}\varphi)^{-1} \cdot \mathscr{L}_{\varphi}^{\otimes 2}(\mathscr{J}f)-(\mathscr{J}f)^2.\nonumber
    \end{eqnarray}
Therefore \begin{eqnarray}
    \Gamma_{\varphi,2}(f)&=&\frac{1}{2}\bigg(\mathscr{L}_{\varphi}^{\otimes 2}((\mathscr{J}\mathscr{D}\varphi)^{-1}\cdot (\mathscr{J}f)^2)-2(\mathscr{J}\mathscr{D}\varphi)^{-1} \cdot \mathscr{L}_{\varphi}^{\otimes 2}(\mathscr{J}f)\cdot \mathscr{J}f +2(\mathscr{J}f)^2\bigg).\nonumber
\end{eqnarray}
\end{proof}
\begin{flushleft}
    However, the calculations required to prove the next conjecture seem very complicated at the moment (especially since we cannot use a spectral approach), so we will leave them for further investigation. Worst of all, even if the conjecture could be proved, there is (to our knowledge) no proof of the free Logarithmic-Sobolev inequality via a non-commutative Bakry-Emery criterion, so we cannot get to the free Klartag estimate for now.
\end{flushleft}
\begin{conjecture}
   Under the above notations and assumptions we have for all smooth functions $(f,g)$:
    \begin{equation}
        \Gamma_{\varphi,2}(f)\geq \frac{1}{2}\Gamma_{\varphi}(f),
    \end{equation}
So, using the conventions of Guionnet and Shlyakhtenko \cite{bak}, we say that
$M_{n.c}:=(K_{\varphi},\mathscr{J}\mathscr{D}\varphi, \nu_{\varphi})$ satisfies $CD(2,\infty)$.
\end{conjecture}
A proof of such an estimate, analogous to Kolesnikov's result \cite{Kol}, would be very useful to understand how far the parallelism can be made between the classical notions of a weighted Riemannian manifold with Hessian matric given by a convex function or from its complexification given by a certain toric K\"ahler manifold and associated with a centred convex body $K$ or equivalently with a centred log-concave measure $e^{-\varphi}dx$, and their free counterparts.
\newline
In fact, we recall how to associate to the uniform measure on a convex set a suitable toric K\"ahler manifold which is both Einstein with $\lambda=\frac{1}{2}$ and Fano with thus a "good" curvature criterion. The construction is as follows.
\bigbreak
Let $\mu$ be the uniform measure on $K$, then we know (see the full details given in Klartag's work \cite{Klart2}) that it is the moment measure of a smooth convex function $\psi$. We denote the complex torus as $\mathbb{T}_{\mathbb{C}}^n:=\mathbb{C}^n/({\sqrt{-1}}\mathbb{Z}^n)$ equipped with its real torus action:
\newline
$t\cdot(x+\sqrt{-1}y)=x+\sqrt{-1}(y+t),\: t\in \mathbb{T}^n:=\mathbb{R}^n/\mathbb{Z}^n,\:x+\sqrt{-1}y\in \mathbb{T}_{\mathbb{C}}^n$. 
\newline
After a suitable complexification to transform the convex function $\psi$ into a complex p.s.h. function (plurisubharmonic: upper-semicontinuous and subharmonic on each complex line) which amounts to set $\psi_{\mathbb{C}}(x+\sqrt{-1}y)=\psi(x),\:x+\sqrt{-1}y\in \mathbb{T}_{\mathbb{C}}^n$ (note also that in this case the function $\psi_{\mathbb{C}}$ is $\mathbb{T}^n$-invariant), the toric Kähler-Einstein manifold $M:=(\mathbb{T}_{\mathbb{C}}^n, \omega_{\psi})$ equipped with its Kähler form $\omega_{\psi}:=2\sqrt{-1}\partial\bar{\partial}\psi=\frac{\sqrt{-1}}{2}\sum_{i,j=1}^n\omega_{i,j}dz_i\wedge d\bar{z}_j$ (where we set as usual $\psi_{i,j}:=\partial^2{\psi}/(\partial x_i\partial x_j)$) has many pleasant properties: the most important one being that its Ricci tensor where $\Ric \omega=-\sqrt{-1}\partial\bar{\partial}\log(\det\psi_{i,j})$ is exactly equals to half the Riemmanian metric: $\Ric \omega_{\psi}=\frac{1}{2}\omega_{\psi}$ (thus a Einstein metric and this is called a toric Fano manifold since its first Chern class is positive, i.e. $c_1(M)>0$) which makes it very similar to a Gaussian space (in particular to obtain Log-Sobolev, Gaussian concentration, isoperimetric inequalities...).
\newline
Moreover, the \textit{moment map} $\mathbb{T}_{\mathbb{C}}^n\ni x+\sqrt{-1}y\mapsto \nabla \psi(x)\in \mathbb{R}^n$ send the volume measure induced by the Riemannian metric, i.e. $\Vol:=w_{\psi}^n/n!=w_{\psi}\wedge\ldots\wedge w_{\psi}/n!$ to the uniform measure $\mu$ on $K$ (see Abreu \cite{abreu}).
\begin{remark}
               We are also grateful to D. Shlyakhtenko, who also pointed out to us that this kind of free SDE's \eqref{car} also arise in the study of "\textit{free dilations}", i.e. $\alpha_t:M\rightarrow \tilde{M}: =M*L(\mathbb{F}_{\infty})$ which are trace preserving $*$-automorphisms such that $\lVert \alpha_t(x)-x\rVert_{2}\underset{t\rightarrow 0}{\rightarrow} 0$ which also solve free SDE's. This has many applications to prove structural properties of von Neumann algebras, especially in the group case, i.e. $M=L(\Gamma)$ where $\Gamma$ is an icc group. Indeed, it was noticed by Dabrowski and Ioana \cite{DabIoa} that the existence of unbounded $1$-cocycle coming from unbounded derivations implies strong restrictions on the structure of the group in order to be able to have a free dilation (based on another earlier work of Dabrowski's \cite{Dab10}). This could be very helpful in order to use the breakthrough of Popa's deformation/rigidity theory to obtain structural properties on these group von Neumann algebras.
\end{remark}
\section{Transporting free Stein kernels to other free Gibbs measure}
In this short section we will show that, as Fathi did in section 5 of \cite{mm}, we are able to construct free Stein kernels with respect to other free Gibbs reference measures in order to compare, for example, $\mu$ with a free Gibbs state $\nu_V$ with potential $V$ smooth and convex, provided it exists. This will be possible as soon as the pushforward measure $\mu_V$ of the measure $\mu$ by $\mathscr{D}V:=V'$ has a barycenter at the origin. It is also, as expected, exactly the free counterpart of the kernel discovered by Fathi in section 5 of \cite{mm}.
\begin{theorem}
Let $\nu_{V}$ be a free Gibbs measure associated with the potential $V$. Suppose $V$ is strictly convex and $\mathcal{C}^2$. Suppose also that we are given a measure $\mu$ which has a density w.r.t. the Lebesgue measure and is supported on a compact interval. Let $\mu_V$ be the pushforward of $\mu$ by $\mathscr{D}V$ and suppose that $\int xd\mu_V(x)=0$. Then
\begin{equation}
(x,y)\mapsto A_V(V'(x),V'(y)).(\mathscr{J}\mathscr{D}V)(x,y))^{-1}:=A_V(V'(x),V'(y))\frac{x-y}{V'(x)-V'(y)}
\end{equation}
is a free Stein kernel for $\mu$ with respect to the potential $V$, where $A_V:\mathbb{R}^2\rightarrow \mathbb{R}$ is an arbitrary free Stein kernel for $\mu_V$ with respect to the standard semicircular potential $\nu_{\frac{1}{2}x^2}$.
\end{theorem}
\begin{proof}
Let's denote $\mu_{V}$ the pushforward of $\mu$ by $V'$. Then we set $g(x)=f((V^*)'(x))$ for an arbitrary test function $f$, and we have:
\begin{eqnarray}\label{gs}
\int_{\mathbb{R}}V'(x)f(x)d\mu(x)&=&\int_{\mathbb{R}}V'(x)g(V'(x))d\mu(x)\nonumber\\
&=&\int_{\mathbb{R}}xg(x)d\mu_{V}(x)\nonumber
\end{eqnarray}
Now let's assume that $\mu_{V}$ admits free Stein kernels with respect to the standard semicircular potential $\nu_{\frac{1}{2}x^2}$. This is always ensured by the result of Cébron, Fathi and Mai in \cite{FCM}, provided that
\begin{equation}
\int xd{\mu_V}=\int_{\mathbb{R}}V'(x)d\mu=0,\nonumber
\end{equation}
\newline
Now denote $A_V:\mathbb{R}^2\rightarrow \mathbb{R}$ such a kernel, then \eqref{gs} becomes
\begin{eqnarray}
\int V'(x)f(x)d\mu&=& \iint A_V(x,y)\frac{g(x)-g(y)}{x-y}d\mu_V(x)d\mu_V(y)\nonumber\\
&=&\iint A_V(V'(x),V'(y))\frac{g(V'(x))-g(V'(y))}{V^{'}(x)-V^{'}(y)}d\mu(x)d\mu(y)\nonumber\\
&=&\iint A_V(V'(x),V'(y))\frac{f(x)-f(y)}{V^{'}(x)-V^{'}(y)}d\mu(x)d\mu(y)\nonumber\\
&=&\iint A_V(V'(x),V'(y))\frac{x-y}{V'(x)-V'(y)}\frac{f(x)-f(y)}{x-y}d\mu(x)d\mu(y),\nonumber
\end{eqnarray}
Where we used the fact that $V''(x)>0$ for $x\in \supp(\mu)$ and thus that $V'(x)-V'(y)\neq 0$, $\mu$ almost everywhere.
\end{proof}\qed

\section{Connexion with the free (weighted) Poincaré inequalities and free diffusions processes}
In the free case, free Poincar\'e inequalities were discovered by Voiculescu in an unpublished note (see lemma 2 in \cite{DAB}) and studied by many authors (see for example Ledoux and Popescu \cite{LP} for a one-dimensional proof and various properties of the Poincaré constant under smooth changes of variables...). Note that these inequalities also hold in the multidimensional case, and have profound consequences for establishing regularity properties of von Neumann algebras, e.g. in the seminal work of Dabrowski \cite{DAB} it was proved that if $X=(x_1,\ldots,x_n)$ has finite free Fisher information, then $W^*(X)$ is a factor which doesn't have the $\Gamma$ property of Murray and von Neumann, and in particular is non-amenable. \footnote{In an unpublished note which Dabrowski personally communicated to the author, Dabrowski actually proved that the absence of the $\Gamma$ property still holds under the weaker assumption of a full non-microstates free entropy dimension: $\delta^*(X)=n$, and expect the same kind of result to hold under the weakest possible condition, i.e. when $\delta^*(X)>1$, as Voiculescu proved for the microstates version of free entropy (see corollary $7. 5$ in \cite{vcartan})} In the one dimensional case, a free Poincar\'e inequality in \textit{the sense of Voiculescu} means that for any compactly supported probability measure $\mu$, there exists a constant $C$ (depending only on $\mu$ and known to be bounded from above by $2\rho^2(\mu)$ where $\rho(\mu)=\sup\left\{\lvert z\rvert, z\in supp(\mu)\right\}$ in full generality and without any restriction on $\mu$) such that for all test functions $f$:
\begin{equation}\label{FPIn}
    Var_{\mu}(f)\leq C\iint\bigg(\frac{f(x)-f(y)}{x-y}\bigg)^2d\mu(x)d\mu(y)
\end{equation}
In fact, the important work of Ledoux and Popescu \cite{LP} illuminates a deep connection between large deviations of random matrices and free functional inequalities, where in this paper the authors prove another type of free Poincaré inequality by a mass transport argument. In fact, this new inequality can be generalised by a new variance estimate involving higher order non-commutative derivatives, see the important work of Houdré and Popescu \cite{Hou}.
\bigbreak
In the classical case, following an idea of Saumard \cite{Saumard} in dimension one, Fathi \cite{mm} proved that in every dimension $d\geq 1$, moment measures with sufficiently smooth moment maps satisfy a \textit{Weighted Poincaré Inequality} with an explicit weight given by the \textit{Moment Stein Kernel} (but only for this particular construction of the Stein Kernel). In the free case, thanks to the free Brascamp-Lieb inequality, we are able to obtain a one-dimensional analogue involving the free moment Stein kernel.
\bigbreak
However, since we have not been able to find a proof of the free Brascamp-Lieb inequality (in the sense of Biane), we propose a simple proof based on the calculations found in the paper by Houdré and Popescu \cite{Hou}, and the same kind of arguments for the proof of the Brascamp-Lieb inequality found by Helffer \cite{Helff} in the classical case.
\begin{proposition}(Free Brascamp-Lieb Inequality)
    Let $V$ be a $\mathcal{C}^2$ and strictly convex potential, and consider its associated free Gibbs measure $\nu_V$, then for smooth functions $f$:
    \begin{equation}
            Var_{\nu_V}(f)\leq \iint(\mathscr{J}\mathscr{D}V)^{-1} (\mathscr{J}f)^2\:d\nu_V^{\otimes 2}
    \end{equation}
    
\end{proposition}
\begin{proof}
    Set $M_V:=-2\partial_x(Id\otimes \nu_V)\mathscr{J}+V'\mathscr{D}$, which is the infinitesimal generator of free Langevin diffusion:
    \begin{equation}
    dX_t=-\mathscr{D}V(X_t)+\sqrt{2}dS_t,\nonumber
    \end{equation} 
    and recall as before that for a sufficiently smooth function, $M_V=\mathscr{J}^*_V\mathscr{J}$ where again $\mathscr{J}_V^*$ denotes the adjoint of $\mathscr{J}$ with respect to the inner product in $L^2(\nu_V)$.
    Then, by proposition $10$ in \cite{Hou}, we have for $f$ of class $\mathcal{C}^2$,
    \begin{equation}
        \langle M_V f,f\rangle_{L^2(\nu_V)}=\iint\bigg(\frac{f(x)-f(y)}{x-y}\bigg)^2d\nu_V(x)d\nu_V(y)\nonumber
    \end{equation}
    In particular, since the right-hand side is a Dirichlet form, which is positive and closable, its generator $M_V$ must be essentially self-adjoint and positive.
    \begin{flushleft}
    We then assume, without loss of generality, that $f\in L^2_0(\nu_V):=\left\{f\in L^2(\nu_V),\: \int fd\nu_V=0\right\}$, considering then the pseudo-inverse $M_V^{-1}$ which is well defined and for which we have $f=M_VM_V^{-1}f$.
     \end{flushleft}
    \begin{flushleft}
        It is not difficult to see that
    \begin{equation}
        Var_{\nu_V}(f)=\iint \mathscr{J}M_V^{-1}f\cdot \mathscr{J}f\:d\nu_V^{\otimes 2}\nonumber
    \end{equation}
    \end{flushleft}
    
   Now consider the linear extension of the operator $M_V^{\otimes 2}:=M_V\otimes id+id\otimes M_V$, which acts on a simple function with values in $\mathbb{R}^2$ as follows
   \begin{equation}
M_V^{\otimes 2}(f\otimes g)=M_Vf\otimes g+ f\otimes M_Vg,\nonumber
   \end{equation}
   where, as usual, $(f\otimes g)(x,y):=f(x)g(y)$.
   \newline
  It is then not difficult to show that this operator $M_V^{\otimes 2}$ is essentially self-adjoint and positive. 
\begin{flushleft}
\bigbreak
  Then we can check that,
   \begin{equation}
       \mathscr{J}M_V=(M_V^{\otimes 2}+\mathscr{J}\mathscr{D}V)\mathscr{J},\nonumber
   \end{equation}
   which follows from
   \begin{equation}
       \mathscr{J}(-2\partial_x(Id\otimes \nu_V)\mathscr{J})=M_V^{\otimes 2},\nonumber
   \end{equation}
\end{flushleft}
Hence, we have: \begin{equation}
       (M_V^{\otimes 2}+\mathscr{J}\mathscr{D}V)^{-1}\mathscr{J}=\mathscr{J}M_V^{-1}\nonumber
   \end{equation}
And finally since $M_V^{\otimes 2}$ is non-negative, we have that $(M_V^{\otimes 2}+\mathscr{J}\mathscr{D}V)^{-1}\leq (\mathscr{J}\mathscr{D}V)^{-1}$, which concludes the proof.
\end{proof}
\qed
\begin{flushleft}
The following lemma refines a particular application of the free Cafarelli-Klartag contraction theorem \ref{cafarelli}, namely the stability of the free Poincar\'e inequality by Lipschtz map. Indeed, for example, it is easy to see that if the (smooth) transport map from the semicircular to a free Gibbs measure $\nu_u$ is associated with $u$ a $1$-uniformly convex potential, then $\nu_u$ satisfies a free Poincar\'e inequality with a free Poincar\'e constant less than $1$, see also Leodux and Popescu, Proposition $2$ in \cite{LP}.
\end{flushleft}
\begin{proposition}\label{prop8}
    Let $\mu$ a centered compactly supported probability measure, with a density w.r.t the Lebesgue measure and with a sufficiently  smooth  moment map $\varphi$ (at least of class $\mathcal{C}^2)$. Consider its moment free Stein kernel 
    $A=\mathscr{J}\mathscr{D}\varphi(\mathscr{D}\varphi^*)$. We then have the following free weighted Poincare inequality valid for smooth functions $f$:
    \end{proposition}
    \begin{eqnarray}
        Var_{\mu}(f)\leq \iint A \cdot (\mathscr{J}f)^2\:d\mu(x)d\mu(y),
    \end{eqnarray}
\begin{proof}
    Let's denote $\varphi$ as the free moment map of $\mu$, and $A=\mathscr{J}\mathscr{D}\varphi(\mathscr{D}\varphi^*)$ its free moment Stein kernel. Then, by using the chain rule for the non-commutative derivative, we get:
    \begin{eqnarray}
        Var_{\mu}(f)&=&Var_{\nu_{\varphi}}(f\circ \varphi')\nonumber\\
       & \leq & \iint (\mathscr{J}\mathscr{D}{\varphi})^{-1}\cdot \mathscr{J}\mathscr{D}\varphi\cdot \mathscr{J}f(\varphi')\cdot \mathscr{J}\mathscr{D}{\varphi}\cdot \mathscr{J}f(\varphi')d\nu_{\varphi}^{\otimes ^2}
       \nonumber\\
       &=& \iint \mathscr{J}\mathscr{D}{\varphi}\cdot (\mathscr{J}f(\varphi'))^2 d\nu_{\varphi}^{\otimes 2}\nonumber\\
       &=& \iint A\cdot (\mathscr{J}f)^2d\mu^{\otimes 2}\nonumber
    \end{eqnarray}
 \qed
\end{proof}
\begin{flushleft}
We now return to the probabilistic interpretation of {\it free moment Stein kernel} and the result \ref{prop8} as a spectral gap for the generator of a free diffusion process. We also continue to denote $(S_t)_{t\geq 0}$ as a free Brownian motion and $(\mathcal{A}_t)_{t\geq 0}$ as its corresponding filtration. We also assume for our exposition that $\varphi$ is of class $\mathcal{C}^2$ and still strictly convex.
\begin{flushleft}
     Indeed, another nice feature of this construction is the positivity and symmetry of these free Stein kernels, i.e. 
$\forall (x,y)\in \mathbb{R}^2, \mathscr{J}\mathscr{D}u(\mathscr{D}u^*(x),\mathscr{D}u^*(y))\geq 0$, from the convexity assumption on $u$, and the symmetry: $\mathscr{J}\mathscr{D}u(\mathscr{D}u^*(x),\mathscr{D}u^*(y))=\mathscr{J}\mathscr{D}u(\mathscr{D}u^*(y),\mathscr{D}u^*(x))$.
     In fact, using this argument, Fathi and Mikulincer \cite{FM} have translated the construction of {\it moment Stein kernels} of the form $A=\Hess\varphi(\nabla \varphi^*)$ into a probabilistic way.
    \bigbreak
That is, under some mild assumptions (e.g. this Stein kernel, which by construction is always symmetric and positive semidefinite, i.e. in $S_d^{+}(\mathbb{R})$, should actually be inside the positive cone and uniformly bounded from below), it is possible to construct a diffusion process with $\mu$ that has its {\it unique invariant measure}, which at first sight might be very useful to do sampling from $\mu$. Specifically, if $(B_t)_{t\geq 0}$ is a standard Brownian motion on some filtered probability space, then the diffusion process is $(X_t)_{t\geq 0}$:
\begin{equation}\label{class}
    dX_t=-X_tdt+\sqrt{2\Hess\varphi(\nabla \varphi)^*}(X_t)dB_t,
\end{equation}
has $\mu$ as a unique invariant measure (uniqueness is then ensured by a standard criterion on the lower bound of this diffusion coefficient). However, as noted by Fathi and Mikulincer \cite{FM}, the \textit{moment Stein kernels} are generally not globally Lipschitz, but generally belong to some Sobolev space (and are almost never explicitly tractable). Using these moment Stein kernels, a variant of the Crippa and De Lellis \cite{CripaD} techniques for studying transport equations, and a proxy condition derived from the crucial Lusin-Lipschitz property, they obtained a new generalisation of the Ledoux, Nourdin, Peccati estimates \cite{LNP} relating transport distances and Stein discrepancies in a non-Gaussian setting, that is for invariant measures $\mu$ of diffusions that are well-conditioned log-concave measures.
\end{flushleft}
\begin{definition}
Let $(X_t^{X_0})_{t\geq 0}$ be the free stochastic process starting from arbitrary (bounded) initial data $X_0\in \mathcal{M}_{s.a}$ with law $\mu$ assumed to be free of free Brownian motion $(S_t)_{t\geq 0}$, and a weak solution of the following free SDE given by Picard iteration if it exists:
\begin{equation}\label{freesde}
dX_t=-X_tdt+(2\mathscr{J}\mathscr{D}\varphi(\mathscr{D}\varphi^*(X_t)))^{\frac{1}{2}}\sharp dS_t,
    \end{equation}
    \end{definition}
    \begin{remark}
This free process is actually a modification of the free Ornstein-Uhlenbeck process:
\begin{equation}
    dX_t=-X_tdt+\sqrt{2}dS_t,
\end{equation}
which has $\mathcal{S}(0,1)$ invariant measure, and for which we recall that we have exponential fast convergence to equilibrium for the relative entropy via the free Log-Sobolev inequality (see Biane \cite{BLS}), and even for the free quadratic Wasserstein metric thanks to the free Talagrand inequality proved by Biane and Voiculescu \cite{BV} (\it{Otto-Villani}-type estimates).
\newline
Here the usual conventions of Langevin diffusion are reversed: all the information is carried in the diffusion coefficient and not in the drift.
\end{remark}
\begin{flushleft}
First we calculate the infinitesimal generator $\mathscr{L}$ of this free SDE \eqref{freesde}. Thanks to the free Ito formula of Biane and Speicher (see also the trace Ito formula of Nikitopolous): Theorem 3.5.3, and remark 3.5.4 in \cite{ENiki}) 
we can easily deduce that the generator is self-adjoint and, for all $\mathcal{C}^2$-bounded functions $f$, is given by
\begin{eqnarray}\label{gener}
    \mathscr{L}f(x)&=&\Delta_{ \mathscr{J}\mathscr{D}u(\mathscr{D}u^*(X_0))}f(x)-xf'(x),\nonumber
\end{eqnarray}
where we denote $\Delta_{ \mathscr{J}\mathscr{D}u(\mathscr{D}u^*(X_0))}$ as the correction term coming from the free Ito formula (it is really important to note here that the generator depends on the law of the initial condition).
\newline
Then we can note that if the free SDE \eqref{freesde} is well defined, then in this case $\mathbb{E}_{\mu}(\mathscr{L}f)=0$ for all $\mathcal{C}^2$ bounded functions (where $\mathbb{E}_{\mu}$ denotes the expectation taken under $\mu$). In other words, $\mu$ is the invariant measure of the free SDE \eqref{freesde}:
\begin{eqnarray}
    \int \mathscr{L}f(x)d\mu(x)=0,\nonumber
\end{eqnarray}
In fact, it is easy to see that this is a self-adjoint operator and that
\begin{eqnarray}
       \int \mathscr{L}f(x)d\mu(x)
    &=&\iint \frac{f'(x)-f'(y)}{x-y} \mathscr{J}\mathscr{D}u(\mathscr{D}u^*(x),\mathscr{D}u^*(y))d\mu(x)d\mu(y)-\int xf'(x)d\mu(x)\nonumber\\
    &=&0.\nonumber
\end{eqnarray}
since it is exactly the Stein equation in its weak formulation that is satisfied by the measure $\mu$ (we assume here that the test functions are gradients, i.e. $f\rightarrow \mathscr{D}f:=f'$).
\end{flushleft}
From now on, the above statement \ref{prop8} can be more conveniently reformulated by saying that the following Dirichlet form has a spectral gap $1$.
\newline
Indeed, for a function $f\in \Dom(\mathscr{L})$, we define the Dirichlet form:
\begin{equation*}
    f\mapsto \mu(f\mathscr{L}f)=\mu\otimes \mu(A\cdot (\mathscr{J}f)^2),
\end{equation*}
The connection with the weighted free Poincaré inequalities can be made more explicit from the above statement by saying that
\begin{equation}
    Var_{\mu}(f)\leq \mu\otimes \mu(A\cdot (\mathscr{J}f)^2),\nonumber
\end{equation} 
In particular, if $\varphi$ is smooth and $\rho$-uniformly convex for some $\rho>0$, then:
\begin{equation}
    Var_{\mu}(f)\leq {\rho}\iint \bigg(\frac{f(x)-f(y)}{x-y}\bigg)^2d\mu(x)d\mu(y).
\end{equation}
\end{flushleft}
\section{A first step towards a multidimensional extension}
In this part we'll explain how our previous construction of free moment Stein kernels can be extended to a multidimensional setting, albeit with strong restrictions since we can only deal with free moment measures that are "close" to the semicircular law. In this case, our results can be interpreted as a first step towards studying regularity estimates on free Monge-Ampère equations and their links to free diffusion operators (and towards understanding the non-commutative Hessian structure associated with free Gibbs laws that are sufficiently close to the semicircular law).
We will first introduce the main notation, which is slightly different now that we are in a strongly non-commutative setting.
\begin{flushleft}
  We first recall what the space of the test function (the free algebra) will be, and how the Banach algebra of non-commutative power series with radius of convergence $R>0$ is defined.  
\end{flushleft}
\begin{definition}
    Let $t_1,\ldots,t_n$ be non-commuting self-adjoint variables. Then we denote by $\mathscr{P}:=\mathbb{C}\langle t_1,\ldots,t_n\rangle$ the algebra of non-commutative polynomials (also called free algebra) in these variables. We also set $\mathscr{P}_0 \subset \mathscr{P}$ as the linear span of polynomials with zero constant term.
\end{definition}
In this space we consider the family of norms $\lVert \cdot\rVert_R$ with $R>0$, defined as follows. For a monomial $q$ and arbitrary $P\in\mathscr{P}$, let $\lambda_q(P)$ denote the coefficient of $q$ in the decomposition of $P$; thus we have
$P=\sum_q\lambda_q(P)q$. Then we set
\begin{equation} 
    \lVert P\rVert_R=\sum_{q:\deg(q)\geq 0}\lvert \lambda_q(P)\rvert R^{\deg(q)},
\end{equation}
we will denote by $\mathscr{P}^{(R)}$ the completion of $\mathscr{P}$ with respect to this norm. This is a Banach algebra, which can be viewed as the algebra of non-commutative power series with a convergence radius of at least $R>0$.
\begin{flushleft}
We now introduce the well-known concept of tracial non-commutative distributions, which will make the notation easier in the sequel.
\begin{definition}
    Let $X=(X_1,\ldots,X_n)$ in a $W^*$ tracial probability space $(\mathcal{M},\tau)$, then by definition the non-commutative distribution (or law) of $X$, denoted $\mu_X$, is the linear form
    \begin{eqnarray}
        \mu_X:&\mathscr{P}&\rightarrow \mathbb{C}\nonumber\\
        &P&\mapsto \tau(P(X)).
    \end{eqnarray}
\end{definition}
\end{flushleft}
\begin{remark}
    In fact, we can easily check that 
    \begin{enumerate}
        \item $\mu_X(1)=1$ (unital),
        \item $\mu_X(P^*P)\geq 0$ (positive),
        \item $\mu_X(PQ)=\mu_X(QP)$ (tracial),
        \item $\mu_X(X_{i_1}\ldots X_{i_n})<R^n$ for any choice of $i_1,\ldots,i_n\in \left\{1,\ldots,n\right\}$ and where $R$ is any positive number such that $R>\underset{i=1,\ldots,n}{\sup}\lVert X_i\rVert$ (exponentially bounded).
    \end{enumerate}
   This motivate to define more generally the notion of non-commutative distribution which is the following.
\end{remark}
\begin{definition}
   A tracial non-commutative distribution is a trace on the $C^*$-universal free product $C([-R,R])^{*n}$ for some $R>0$. This space is denoted $\Sigma_{m,R}$ and is endowed with the weak $*$-topology, i.e. the topology of pointwise convergence on $\mathscr{P}$. Equivalently, an element of $\Sigma_{m,R}$ is a unital, positive, tracial and exponentially bounded map.
\end{definition}
Let's also recall Voiculescu's definitions of free difference calculus, which are the cyclic gradients and the free difference quotients. In the following definitions we first define these operators on the space of non-commutative polynomials, and any tensor product that appears must be understood as the algebraic tensor product (i.e. without completion).
\begin{definition}[Voiculescu, \cite{Voic1}]
We define the cyclic derivative on monomials $m\in \mathscr{P}$ as
\begin{equation}
    Dm=(D_1m,\ldots ,D_nm),\nonumber
\end{equation}
where
\begin{equation}
D_jm=\sum_{m=at_{j}b}ba,\nonumber
\end{equation}
and then extended it linearly to $\mathscr{P}$.
\end{definition}
\begin{definition}(Voiculescu section 3 in \cite{V})
The $j$-free difference quotient is defined as:
\begin{equation}
    \partial_jm=\sum_{m=at_jb}a\otimes b^{op},\nonumber
\end{equation}
and then linearly extended to $\mathscr{P}$.
\end{definition}
\begin{flushleft}
    In fact, we can relate the free difference quotients and the cyclic derivatives as follows
\end{flushleft}

\begin{equation}
    D_j=m\circ(\partial_j)^{\sigma},
\end{equation}
where we denote $^{\sigma}$ as the flip homomorphism, i.e. $Q=a\otimes b\in \mathscr{P}\otimes \mathscr{P}^{op}$, $(Q)^{\sigma}=b\otimes a$. This operator is then extended to include linearity.
\begin{definition}\label{2}
For $P=(p_1,\ldots,p_n)$ we define the non-commutative Jacobian as :
\begin{equation*}
\mathscr{J}P = 
\begin{pmatrix}
\partial_1 p_1 & \partial_2 p_1 & \cdots & \partial_n p_n\\
\vdots  & \vdots  & \ddots & \vdots  \\
\partial_1 p_n & \partial_2 p_n & \cdots & \partial_n p_n
\end{pmatrix} \in M_n(\mathscr{P}\otimes \mathscr{P}^{op}),\nonumber
\end{equation*}
\end{definition}
Fortunately, the non-commutative Jacobian enjoys a fundamental chain rule property (up to natural operations).
\begin{flushleft}
    We also introduce the \textit{Number} operator $\mathscr{N}$ and the symmetrization operator $\mathscr{S}$, which are defined on $\mathscr{P}$ by
    \begin{equation}
        \mathscr{N}(t_{i_1}\ldots t_{i_p})=pt_{i_1}\ldots t_{i_p}
    \end{equation}
    i.e. $\mathscr{N}P=\sum_{i=1}^p\partial_jP\sharp t_j$, this is the linear operator that multiplies a monomial of degree $p$ by $p$.
    \bigbreak
    The symmetrization operator $S$ is defined on monomials in $\mathscr{P}_0$ as 
    \begin{equation}
        \mathscr{S}t_{i_1}\ldots t_{i_p}=\frac{1}{p}
\sum_{r=1}^p t_{i_{r+1}}\ldots t_{i_p}t_{i_1}\ldots t_{i_r}.  \end{equation}
\end{flushleft}
\begin{flushleft}
    Using corollary $3.3$ in \cite{NZi}, it can be shown that the operator $\mathscr{J}\mathscr{D}$ extends to $\mathscr{P}^{(R)}$ for any $R>0$. Here $\mathscr{P}\hat\otimes_R \mathscr{P}^{op}$ is the completion of the algebraic tensor product $\mathscr{P}^{(R)}\otimes (\mathscr{P}^{(R)})^{op}$ with respect to the projective tensor norm (e.g. \cite[Section 3.1]{NZi} for further details).
\end{flushleft}
\bigbreak
\begin{flushleft}
    We now introduce the notion of the free Gibbs law in the multivariate setting, which is similar to the univariate one, except that here, since there are different notions of entropy, we'll focus on the microstates one, which we'll recall in what follows.
    \end{flushleft}
    \bigbreak
    \begin{definition}(Voiculescu, Microstates free entropy \cite{Ventro})
    \newline
    Let $X=(x_1,\ldots,x_n)\in (\mathcal{M},\tau)^n$ be a $W^*$ tracial probability space. For each $k\in \mathbb{N}$, let $M^{sa}_k$ denote the set of $k\times k$ self-adjoint matrices equipped with its non-normalised trace $\Tr$. For $l\in \mathbb{N}$ and $\varepsilon>0$, let 
    \begin{equation}
        \Gamma(X;k,l,\varepsilon):=\left\{Y\in M^{sa}_k,\bigg\lvert \frac{1}{n}Tr(P(Y))-\tau(P(X))\bigg\rvert<\varepsilon,\:\forall P\in \mathscr{P}\:\mbox{with\:} \deg(P)\leq l\right\}.
    \end{equation}
    Then the free entropy of the microstates of $X$ is defined as follows
    \begin{equation}
        \chi(X):=\underset{l,\varepsilon}{\inf}\:\underset{k\rightarrow \infty}{\limsup}\frac{1}{k^2}\log(\Vol(\Gamma(X;k,l,\varepsilon)))+\frac{n}{2}\log k.
    \end{equation}
    Note that we also use a shortcut: $\chi(\tau_X)=\chi(X)$. 
    \begin{flushleft}
        More generally, any non-commutative distribution $\tau$ with exponential bound $R>0$ can be realised as the non-commutative distribution of some $n$-tuple of self-adjoint operators $X=(X_1,\ldots,X_n)$ with $\lVert X\rVert_{\infty}:=\underset{i=1,\ldots, n}{\max} \lVert X_i\rVert_{\infty}\leq R$ on a $W^*$ tracial probability space. This is a version of the Gelfand-Naimark-Segal construction (see Proposition $5.2.14$ in \cite{AGZ}).
    \end{flushleft}
    
    \end{definition}
    This leads to the following definition, which is well known and was introduced by Voiculescu in \cite{VF} and Guionnet and Shlyakhtenko in \cite{GS}. 
    \begin{definition}
        The free Gibbs law $\tau_U$ associated with the potential $U$ is the minimiser of $-\chi(\tau_U)+\tau(U(X_1,\ldots,X_n)$ if it exists.
    \end{definition}
We are now in a position to give a proper definition of free moment laws in a multivariate setting.
\begin{definition}
    The law of a non-commutative tuple of random variables $X=(X_1,\ldots,X_n)$ is called a free moment law if there exists a self-adjoint non-commutative power series $U$ such that the law $\tau_U$ is well defined, and the law of non-commutative random variables $Y=(Y_1,\ldots,Y_n)$ such that
    \begin{equation}
        X=(X_1,\ldots,X_n)=(\mathscr{D}U)(Y_1,\ldots,Y_n),
    \end{equation}
    that is, for all $j=1,\ldots,n$,
    \begin{equation}
        X_j=\mathscr{D}_jU(X_1,\ldots,X_n).
    \end{equation}
\end{definition}
\begin{flushleft}
    Then the definition of the free Stein kernels in the multivariate setting takes the following form
\end{flushleft}

\begin{definition}
A free Stein kernel for an $n$-tuple $X\in (\mathcal{M},\tau)^n$ with respect to a potential $V\in \mathscr{P}$ is an element of $L^2(M_n(\mathcal{M}\bar{\otimes} \mathcal{M}^{op}), (\tau \otimes \tau^{op}) \circ Tr)$ such that for any $P \in \mathscr{P} ^n$:
\begin{equation}
    \langle [\mathscr{D}V](X),P(X)\rangle_{\tau}=\langle A,[\mathscr{J}P](X)\rangle_{\tau \otimes \tau^{op}},
\end{equation}
\end{definition}
\begin{flushleft}
\end{flushleft}
\begin{equation}
\Sigma ^*(X|V)=\inf_{A}\lVert A-(1\otimes 1^{op})\otimes I_n\rVert_{L^2(M_n(\mathcal{M}\bar{\otimes} \mathcal{M}^{op}),(\tau \otimes \tau^{op}) \circ Tr)},
\end{equation}
where the infinimum is taken over all admissible Stein kernels $A$ of $X$ relative to $V$.
\bigbreak
\begin{flushleft}
    We also need to introduce an appropriate notion of convexity, which was already precisely defined in the work of Guionnet and Shlyakhtenko \cite{GS} and was also used by Fathi and Nelson \cite{FN} to prove a microstates variant of the free Log-Sobolev inequality.
\end{flushleft}
\begin{definition}
    We say that $f\in \mathscr{P}^{(R)},\: R>0$ is convex if and only if $\mathscr{J}\mathscr{D}f=Q^*Q$ for some $Q\in M_n(\mathscr{P}^{(R)}\hat{\otimes} \mathscr{P}^{(R)})$. In shorthand we will refer to this as $\mathscr{J}\mathscr{D}f\geq 0$.
\end{definition}
\begin{flushleft}
    The following theorem is stated in a more general setting, since we don't assume that the target-free Gibbs measure is necessarily "close" to the semicircular law, \end{flushleft}
\begin{theorem}\label{multivstein}
    Let $X=(X_1,\ldots,X_n)$ be a free moment law associated with a non-commutative power series $U$ in $\mathscr{P}^{(R)}$ for some $R>0$, i.e, such that there exists $Y=(Y_1,\ldots,Y_n)$ of the law $\tau_U$ which is well defined and such that for all $j=1,\ldots,n$, $X_j=\mathscr{D}_jU(Y_1,\ldots,Y_n)$.
    \begin{flushleft}
        Now suppose $\mathscr{D}U=(\mathscr{D}_1U,\ldots, \mathscr{D}_nU)\in (\mathscr{P}^{(R)})^n$ has a compositional inverse $G\in (\mathscr{P}^{(\lVert \mathscr{D}U\rVert_R)})^n$ with $\mathscr{J}\mathscr{D}U\geq 0$, then
    $A=[\mathscr{J}\mathscr{D}U](\mathscr{D}G(X))$ is a free Stein kernel for $X$ with respect to the standard semicircular potential.
   $V_1=\frac{1}{2}\sum_{i=1}^n X_i^2$ which is also positive in $M_n(W^*(X)\bar{\otimes}W^*(X))$.\end{flushleft}
    
\end{theorem}
\begin{proof}
    By definition, since $Y$ follows the free Gibbs law associated with $U$, it satisfies the following Schwinger-Dyson equation: $\forall P=(P_1,\ldots,P_n)\in \mathscr{P}^n$.
   \begin{eqnarray}
        \tau([\mathscr{D}U](Y)\cdot P(Y))=\tau\otimes\tau\circ \Tr([\mathscr{J}P](Y)),
    \end{eqnarray}
    So by setting $Q=P(\mathscr{D}U):=\bigg(P_i(\mathscr{D}_1U,\ldots,\mathscr{D}_nU)\bigg)_{i=1}^n$ component-wise, by the trace property of the state and the chain rule for the non-commutative Jacobian, we get
       
    \begin{eqnarray}
          \tau(P(\mathscr{D}U)(Y)\cdot [\mathscr{D}U](Y))&=&\tau\otimes\tau\circ \Tr([\mathscr{J}P(\mathscr{D}U)](Y))\nonumber\\
          &=&\tau\otimes\tau\circ \Tr([\mathscr{J}{\mathscr{D}}U](Y)\sharp[\mathscr{J}P](\mathscr{D}U(Y)).\nonumber
          \end{eqnarray}
Now, by changing the variables: $Y=(Y_1,\ldots,Y_n)=\mathscr{D}G(X_1,\ldots,Y_n)$, that $\mathscr{D}U(\mathscr{D}G(X))=X$, we come to
\begin{equation}
    \tau(X\cdot P(X))=\tau\otimes\tau\circ \Tr([\mathscr{J}\mathscr{D}U](\mathscr{D}G(X))\sharp [\mathscr{J}P](X)).\nonumber
\end{equation}
which is exactly the desired conclusion. The positivity of such a free Stein kernel then follows easily from the convexity assumption on $U$.
\end{proof}
\qed
\begin{flushleft}
    We now introduce the following free version of the quadratic Wasserstein distance, which was first introduced in the important work of Biane and Voiculescu \cite{BV}.
\end{flushleft}
\begin{definition}(Biane Voiculescu, \cite{BV})
The free quadratic Wasserstein distance is defined as the following infinimum over couplings with respect to the quadratic cost
\begin{eqnarray*}
    W_2((X_1,\ldots ,X_n),(Y_1,\ldots ,Y_n))&=&\inf \Bigl\{\lVert (X_i^{'} - Y_i^{'})_{1\leq i\leq n}\rVert_2 /
    (X_1^{'},\ldots ,X_n^{'},Y_1^{'},\ldots ,Y_n^{'}) \subset (M_3,\tau)\nonumber\\
    &&(X_1^{'},\ldots ,X_n^{'})\simeq(X_1,\ldots ,X_n),(Y_1^{'},\ldots ,Y_n^{'})\simeq(Y_1,\ldots ,Y_n)\Bigl\}
, \end{eqnarray*}
where $\simeq$ means equality in the $*$-distribution, where $(M_3,\tau)$ is a $W^*$-tracial probability space with each $(X_i^{'},Y_i^{'})\in M_3$, where $\simeq$ means equality in the distribution.
\end{definition}
\bigbreak
\begin{remark}
    We can also define the classical Wasserstein distance between two bounded tuples of random vectors by requiring that the coupling $(X_1',\ldots,X_n',Y_1',\ldots,Y_n')$ lives in an abelian $W^*$-tracial probability space.
\end{remark}
\begin{flushleft}
    Thus, as in the one-dimensional case, via the multidimensional free Wasserstein Stein discrepancy of Cébron \cite{C}, we can obtain the following stability estimates on the free Monge-Ampère equation of Guionnet and Shlyakhtenko, but first, following the work of Fathi, Cébron and Mai \cite{FCM}, we introduce the multidimensional version of the free Sobolev spaces associated with a non-commutative distribution.
\end{flushleft}
\begin{definition}
    For any (tracial or not in full generality) non-commutative distribution $\mu$, we denote $H^1(\mu)$ as the separation-completion of $\mathscr{P}^n$ with respect to the sesqui-linear form:
    \begin{equation}
       \langle P,Q\rangle_{H^1(\mu)}:=\mu\otimes \mu\circ\Tr((\mathscr{J}Q)^*\sharp \mathscr{J}P),\: P,Q \in \mathscr{P}^n.
    \end{equation}
\end{definition}
\begin{remark}
   If we specify the non-commutative law as the distribution $\mu_X$ of a tuple of self-adjoint operators $X=(X_1,\ldots,X_n)$, it is interesting to note that the isotropic cone which is defined as $\mathcal{N}_{\mu_X}:=\left\{P\in \mathscr{P}^n,\: \langle P,P\rangle_{H^1(\mu_X)}=0\right\}$ could contain non-constant polynomials due to the possible algebraic relations between the $X_i$'s. However, if we make some mild assumption on the $n$-tuple $X$, e.g. finite free Fisher information, finite free entropy or even the weakest condition of full non-microstate free entropy, i.e. $\delta^*(X)=n$, then the absence of algebraic relations is automatically satisfied, cf. Charleworth-Shlyakhtenko \cite{CS}, or Mai, Speicher and Weber \cite{MSW}. \end{remark}

    Now we recall that in the seminal paper by Guionnet and Shlyakhtenko \cite{GS}, the authors proved that existence of an analytic and invertible transport $U$ between the free Gibbs state associated with the quadratic potential (semicircular system) and a small perturbation of this potential, i.e. such that $\mathcal{D}U(S_1,\ldots,S_n)$ has the free Gibbs law $\tau_W$ associated with the potential $W$ when $(S_1,\ldots,S_n)$ are free semicircular variables exists and is unique (in an appropriate sense) if $\lVert W-V_1\rVert_{R}$ is small enough. In particular, such a transport solves the following free Monge-Ampère equation.
\begin{lemma}(Guionnet and Shlyakhtenko, Lemma 3.3 in \cite{GS})
\newline
    A free transport $U$ between the free Gibbs laws with potential $V_1=\frac{1}{2}\sum X_j^2$ and potential $W$ exists if and only if (denoting in a shorthand $\tau:=\tau_{V_1}$ as the standard semicircular law):
    \begin{equation}
        (1\otimes \tau+\tau\otimes 1)\Tr\log\mathscr{J}\mathscr{D}U=\mathscr{S}\left\{W(\mathscr{D}U(X))-\frac{1}{2}\sum X_j^2\right\}.
    \end{equation}
    where $\mathscr{S}$ means equality modulo cyclic symmetrization (or commutator).\end{lemma}
    \begin{flushleft}
        Using this notion of free Sobolev spaces, we can now reinterpret in a more pleasant way the quantitative stability result (which here is perturbative by nature) on the free Monge-Amp\'ere equation.
    \end{flushleft}
\begin{corollary}
    Let $X=(X_1,\ldots,X_n)$ be an $n$-tuple of bounded operators as in the previous theorem \ref{multivstein} and $S=(S_1,\ldots,S_n)$ be a standard semicircular system, both living in some $W^*$-tracial probability space. Then we have
    \begin{eqnarray}
        W_2(X,S)^2&\leq& \lVert [\mathscr{J}\mathscr{D}U](\mathscr{D}G(X))-(1\otimes 1)\otimes I_n\rVert_{L^2(M_n(W^*(X)\bar{\otimes} W^*(X))}
        \nonumber\\
        &=&\lVert [\mathscr{J}\mathscr{D}U](Y)-(1\otimes 1)\otimes I_n\rVert_{L^2(M_n(W^*(Y)\bar{\otimes}W^*(Y))}\nonumber\\
        &=&\lVert\mathscr{D}U-Id_{\mathscr{P}^n}\rVert_{H^1(\mu_Y)}.
    \end{eqnarray}
    where $Id_{\mathscr{P}^n}=(t_1,\ldots,t_n)$.
\end{corollary}
In a perturbative regime, that is for free Gibbs near the semicircular law, Bahr and Boschert have been able to extend the construction in a multidimensional case. The method they employed is not variational and relies on a study of an analogous free Monge-Ampère PDE to conclude via a fixed point argument.
\begin{theorem}(Bahr and Boschert, Theorem 3.1 in \cite{BB})\label{th11}
    There exists a $C>0$ and a $\epsilon>0$ such that if $W(X_1,\ldots,X_n)$ is a self-adjoint n.c power series containing only terms of even degree, and $\lVert W\rVert_C<\epsilon$, then there is a corresponding power series $V(Y_1,\ldots,Y_n)$ such that, when $Y$ has the free Gibbs law associated with $\frac{1}{2}\lVert X\rVert^2+V$, then $Y+\mathscr{D}_YV(Y)$ has free Gibbs law $\frac{1}{2}\lVert X\rVert^2+W$.
\end{theorem}
\begin{flushleft}
\textbf{Assumptions:}
\newline
From now on and in the sequel, we assume that given $C>0$ and a $\epsilon>0$ such that theorem \ref{th11} is satisfied. In particular by Proposition $3. 2$ in \cite{BB}, for this choice of cutoff $C$ we know that the free Gibbs law $\tau_{\tilde{W}}$ (with $\lVert W\rVert_{C}<\epsilon$) to the Schwinger-Dyson equation of the potential $\tilde{W}$ exits, and which by theorem \ref{th11} is also the non-commutative law of $Y+\mathscr{D}_YV(Y)$, where we denote $\tilde{V} =V_1+V=\frac{1}{2}\sum_{i=1}^n X_i^2+V$, and $\tilde{W}:=V_1+W=\frac{1}{2}\sum_{i=1}^n X_i^2+W$. 
\newline
We also assume that this cutoff $C>0$ is large enough so that the following non-commutative Hessian metric $U=\mathscr{J}\mathscr{D}\tilde{V}$ and its inverse $U^{-1}=(\mathscr{J}\mathscr{D}\tilde{V})^{-1}$ exists in $M_n(\mathscr{P}^{(C)}\hat{\otimes} \mathscr{P}^{(C)})$.
\bigbreak
    In particular, Bahr and Boschert then prove that $V$ then solves the following variant of the Monge-Ampère equation. 
\end{flushleft}

\begin{lemma}(Bahr and Boschert, lemma 3.3 in \cite{BB})\label{lma3}
\newline
    The non-commutative power series solution of theorem \ref{th11} satisfies
    \begin{equation}\label{bbm}
        (1\otimes \tau_{\tilde{V}}+\tau_{\tilde{V}}\otimes 1)\Tr\log(1+\mathscr{J}\mathscr{D}V)=\mathscr{S}\left\{W(Y+\mathscr{D}V)+(\mathscr{N}-id)V+\frac{1}{2}\lvert \mathscr{D}V\rvert^2\right\}.
    \end{equation}
    where here $\mathscr{S}$ means equality modulo cyclic symmetrization (or commutator).\end{lemma}
    
    \begin{flushleft}
    Since we want to adopt the diffusion point of view exposed for the one-dimensional case, we will linearise this equation and show that the free moment map (modulo commutator and replacement by a non-flat Laplacian) is still an eigenvector of the Laplacian on a non-commutative Hessian manifold, which is always associated with the eigenvalue $-1$.
    \end{flushleft}
    \bigbreak
In the following we rewrite the equation \eqref{bbm} with $\tilde{V}=V_1+V$ and in particular (omitting the superscript $Y$ as the context is clear) that $Y+\mathscr{D}V=\mathscr{D}\tilde{V}$ and $\mathscr{J}\mathscr{D}\tilde{V}=1+\mathscr{J}\mathscr{D}V$.
\bigbreak
Let's now introduce the following operator, which will be the Laplacian part of the free Monge-Ampère equation and will allow us to reformulate it as a free diffusion equation. 
\begin{definition}
For a cutoff $C>0$, with $W\in \mathscr{P}^{(C)}$ and $\epsilon>0$, chosen small enough so that both $\lVert W\rVert_{C}<\epsilon$ and the above statements holds.
\newline
Then, for an arbitrary non-commutative power series $g\in \mathscr{P}^{(C)}$, we set 
    \begin{eqnarray}
    \mathscr{L}_{\tilde{V}}^{\tilde{W}}g&=&2\sum_{i,j=1}^n\: m\circ(id\otimes \tau_{\tilde{V}}\otimes id)(1\otimes (\mathscr{J}\mathscr{D}\tilde{V})^{-1}\sharp (\partial_i\otimes id\circ \partial_j g)\sharp  (\mathscr{J}\mathscr{D}\tilde{V})^{-1}\otimes 1)\nonumber\\
    &-&\sum_{i=1}^n\partial_i\tilde{W}(g)\sharp \mathscr{D}_i\tilde{V}.
\end{eqnarray}
where $m$ is the standard multiplication operator and $(\partial_i\otimes id)\circ \partial_i$ denotes the mixed second-order derivative. We recall that the free difference quotients are coassociative derivatives, i.e. $(\partial_i\otimes id)\circ \partial_j=(id\otimes \partial_j)\circ\partial_i$, see the important work of Dabrowski, Guionnet and Shlyakhtenko in \cite{Dab10,YGS}, and Biane and Speicher \cite{BS} for further details.
\begin{remark}
    The first term of the operator above is in fact the "flat Laplacian".
    In the case where $\tilde{W}=V_1$ (i.e. $W=0$), we in fact are left with the infinitesimal generator of free Langevin diffusion with potential $V$
\begin{equation}
    dX_t=-[\mathscr{D}V](X_t)+\sqrt{2}dS_t.\nonumber
\end{equation}
where here $S_t=(S_t^{(1)},\ldots, S_t^{(n)})_{t\geq 0}$ is a $n$-dimensional free Brownian motion.
\end{remark}
    
\end{definition}

\begin{flushleft}
We can now "differentiate" this equation \eqref{bbm} (taking the full cyclic gradient in every direction) to finally write it as a free diffusion equation by introducing only these diffusions operators as follows
    \begin{equation}
        L_{\tilde{V}}^{\tilde{W}}\mathscr{D}_i\tilde{V}+R\mathscr{D}_i\tilde{V}=-\mathscr{S}\mathscr{D}_i\tilde{V}.
    \end{equation}
    where $R$ is a second order term (whose expression can be computed explicitly), which we will remove by modifying the Laplacian using a new free differential calculus for a bigger set of test functions as discovered by Dabrowski, Guionnet in \cite{YGS}.
\end{flushleft}
Indeed, inspired by Dabrowski, Guionnet and Shlyakhtenko's work on free transport for convex potentials \cite{YGS}, we set the following deformed Laplacian, which was introduced to remove this additional second order term $R$ appearing in the free Monge-Ampère equation. The procedure basically amounts to enlarge the set of test functions (to the space of functions with expectations) and change the generator of the diffusion to involve differentiation under expectation. The reader might consult the important heuristic details given in the introduction of \cite{YGS}. It is a crucial fact and it will help to see that free diffusion equations are indeed semigroups with generators given by:
\begin{equation}
    \Delta_{\tilde{W}}^{\tilde V}=\mathscr{L}_{\tilde{W}}^{\tilde V}+\delta_{\tilde{W}}^{\tilde V}.\nonumber
\end{equation}
where $\delta_{\tilde{W}}^{\tilde V}$ is the derivation defined on the following space of non-commutative polynomials with expectations already considered by C\'ebron in \cite{ceb2}
$\mathbb{C}\left\{X_1,\ldots,X_n\right\}:=\mathbb{C}\langle X_1,\ldots,X_n\rangle \otimes \mathscr{S}(\mathbb{C}\langle X_1,\ldots,X_n\rangle)$, where here we let $\mathscr{S}(\mathbb{C}\langle X_1,\ldots,X_n\rangle)$ denotes the symmetric algebra of non-commutative polynomials (see appendix in \cite{ceb2} for furthers details on this construction) by
\begin{equation}
    \delta_{\tilde{W}}^{\tilde {V}}(P_0)=0,\: \delta_{\tilde{W}}^{\tilde {V}}(\tau(P))=\tau(\mathscr{L}_{\tilde{W}}^{\tilde {V}}(P_i)).\nonumber
\end{equation}
The action of this operator can then be easily extended to consider non-commutative power series (analytic functions) $\mathscr{P}^{(R)}$, for some $R>0$  with expectations, i.e. for analytic functions of the form
\begin{equation}
    \sum_{P_0,\ldots,P_n\in \mathscr{P}}a_{P_0,\ldots,P_n}P_0\tau(P_1)\ldots\tau(P_n).\nonumber
\end{equation}
which is exactly the form of a free transport map found by Guionnet and Shlyakhtenko \cite{GS}.
\bigbreak
Since we want to solve the free Monge-Ampère equation in gradient form, we have to define a cyclic gradient that respects the property of differentiation under expectation. In fact, the solution is given by introducing the full cyclic gradient, which is then defined on this space of analytic functions with expectation by setting:
\begin{equation}
    \mathscr{D}_i(P_0\tau(P_1)\ldots\tau(P_n))=\sum_{j=0}^n\mathscr{D}_iP_j\prod_{k\neq j}\tau(P_k).\nonumber
\end{equation}
The action of the full free difference quotient is also needed as in is involve sin the transport equation and it is given as follows, 
\begin{equation}
\partial_i(P_0\tau(P_1)\ldots\tau(P_n))=\partial_iP_0\tau(P_1)\ldots\tau(P_n).\nonumber
\end{equation}
Moreover, we have the commutation relation between the \textit{flat Laplacian} adapted to the free diffusion and the full cyclic gradient, i.e setting \begin{equation}
    \Delta(\cdot)=2\sum_{i,j=1}^n\: m\circ(id\otimes \tau_{\tilde{V}}\otimes id)\partial_i\otimes id\circ \partial_j(\cdot),
\end{equation}
we have
\begin{equation}
    \Delta\mathscr{D}_i=\mathscr{D}_i\Delta, \:\forall i=1,\ldots,n.
\end{equation}
See for example Dabrowski, Guionnet and Shlyakhtenko, equation $(14)$ in \cite{YGS} to note this fact. 
\begin{flushleft}
    Then we can finally check that modulo commutators, we have that in each direction $\mathscr{D}_iV$ is an eigenvalue of the following free diffusion operator.
\begin{proposition}
Each $\mathscr{D}_i\tilde{V}, i=1,\ldots,n$ is an eigenvalue (modulo) commutator of the free Laplacian.
\begin{equation}\label{81}
     \Delta_{\tilde{W}}^{\tilde {V}}(\mathscr{D}_i\tilde{V})=-\mathscr{S}\mathscr{D}_i\tilde{V},\:\forall\:i=1,\ldots,n.
\end{equation}
Moreover, the associated eigenspace is composed (modulo) commutators by all linear functions.
\end{proposition}
\end{flushleft}
\bigbreak
\begin{flushleft}
    In particular, the classical analogue of the important equation \eqref{81}, which plays a key role in many results, can be used to recover (at least asymptotically for large $n$) a central phenomenon in convex geometry, formally proved in Klartag's seminal work 
\cite{convex}, which states that:
\newline
"\textit{High-dimensional, centered, isotropic convex bodies (or isotropic log-concave distributions) have most of their one-dimensional marginal that are approximately standard Gaussian.}" 
\bigbreak
In fact, let's assume that $X$ is a random vector which is uniformly distributed on the convex body $K$ (with non-empty interior) supposed to be an isotropic body of volume one. Then, for a subset $\Theta\subset\mathbb{S}^{n-1}$ with large spherical measure, $\forall \theta \in \Theta$, $\langle X,\theta\rangle\approx\mathcal{N}(0,1)$ if $n$ is sufficiently large. 
This very important result for the class of isotropic, convex bodies can indeed be explained by the previous view of moment maps. This was remarkably pointed out by Fathi as a conjecture about "CLTs for random eigenfunctions in positive curvature" in his \textit{Habilitation} thesis, conjecture $2$ in \cite{hdr}, in order to provide a unified framework for different kinds of (classical) CLTs. 
\bigbreak
Indeed, if we denote $\mu$ as a uniform measure on the isotropic convex body $K$, then if the barycenter of $K$ is at the origin, we know that there exists a moment map $\varphi$ such that $\mu=(\nabla\varphi)_{\sharp}e^{-\varphi}dx$, and considering the weighted Riemannian manifold $(\mathbb{R}^n,\Hess \varphi,e^{-\varphi}dx)$, we know that $\partial_i{\varphi}:=\partial_{e_i}\varphi$ is an eigenvalue of the Laplacian on this manifold, i.e. $\mathcal{L}_{\varphi}(\partial_i{\varphi})=-\partial_i{\varphi}$ (see, for example, Fathi's proof of Theorem $3. 2$ in \cite{mm}), so we can explain Klartag's result by saying that for most directions the pushforward of $e^{-\varphi}$ by $\partial_i \varphi$ is approximately Gaussian (this is also even true for multidimensional vectors if, for example, the convex Body is moreover unconditional, see theorem $1.3$ of Klartag \cite{convex}). The curvature criterion proved by Kolesnikov \cite{Kol}, which ensures that the Ricci curvature is always bounded from below by $1/2$, uniformly, with respect to the dimension $n$ of the ambient Euclidean space, is here a really important feature in Fathi's conjecture, which is motivated by several examples, such as Meckes CLT for the eigfunctions of the Laplacian on the sphere (which has positive constant curvature) \cite{Meckes}. So if Fathi's conjecture about  CLTs for random eigenfunctions in positive curvature is true, then Klartag's theorem would follow.
\newline
It is also important to note that there are other CLTs that exploit this structure: a version of Nualart and Peccati's fourth moment theorem \cite{NP}, or even Ledoux's fourth moment theorem, which generalises this last one for more general structures given by the chaos of a Markov operator \cite{Led}. The main goal of this conjecture is to give a unified geometric framework covering all these different CLTs, and in particular Klartag's central limit theorem for convex bodies \cite{convex}.
\bigbreak
This is the same phenomenon that appears here in the free case, mostly that the free version of this moment measure is such that if $Y$ has the free Gibbs law with potential $\frac{1}{2}\lvert X\rvert^2+V(X)$, then the moment measure $Y+\mathscr{D}_iV(Y)$, which is the free Gibbs law of the potential $\frac{1}{2}\lvert X\rvert^2+V(X)$, has most of its marginals approximately semicircular, i.e. e.g. $\forall i=1,\ldots,n,\: Y_i+\mathscr{D}_iV(Y)\approx S\simeq \sigma$ and most of the coordinates are "\textit{\textbf{asymptotically free}}" from each other: this is the most important remark as it is well known (see corollary $6.7$ \cite{Ventro} of Voiculescu) that freeness and regularity, i.e. a free family of random variables $X_1,\ldots,X_n$ generating a $II_1$ factor, necessarily generates a (possibly interpolated) free group factor. In fact according to Klartag, we can rephrase the convexity as being almost as good as independence (here free independence). Furthermore, we conjecture that the Ricci curvature defined in the same way as in \ref{def7} (but in a multidimensional setting), is bounded from below by $1/2$. This seems to be the reason why, heuristically, for $n$ large enough, we have the isormorphism between these perturbed semicircular von Neumann algebras generated by the $X_i=Y_i+\mathscr{D}_iV(Y), i=1,\ldots, n\leq \infty$ and the free group factor $W^*(S_1,\ldots,S_n)\simeq L(\mathbb{F}_n)$ (this even seems to explain the remarkable non-commutative triangular transport discovered by Jekel \cite{jektriang}). If the construction of moment maps could be extended to more general potentials (and not just the even perturbative case), this would also explain why the isomorphism still holds in the uniformly convex setting, as proved by Dabrowski, Guionnet and Shlyakhtenko \cite{YGS}, at least for a large number of generators.
\end{flushleft}

\begin{flushleft}
All this remarks gives a strong hint that the conjecture of Fathi and Nelson \cite{FN} may be true.
In the following we denote for $a,b\in \mathcal{M}$, a von Neumann algebra equipped with a faithful normal state (not necessarily tracial), $(a\otimes b^{op})^{\dagger}=b^*\otimes (a^*)^{op}$ and for a matrix $A\in M_n(\mathcal{M}\bar{\otimes} \mathcal{M}^{op})$, we denote $A^{\dagger}\in M_n(\mathcal{M}\bar{\otimes} \mathcal{M}^{op})$ such that $[A^{\dagger}]_{i,j}=[A]_{i,j}^{\dagger}$.
\begin{conjecture}(Fathi, Nelson \cite{FN})
  Let $X=(X_1,\ldots,X_n)$ with joint law $\mu_X$ and $A=A^{\dagger}$ a free Stein kernel with respect to $V_1$, then there exists $\epsilon>0$ such that if $\Sigma^*(X|V_0)<\epsilon$, then there exists a free transport from $\mu_X$ to the semicircular law. In particular, $W^*(X)\hookrightarrow L(\mathbb{F}_n)$.
\end{conjecture}
\end{flushleft}
\section{Remarks and open questions}
    The computations done in section \ref{sect4} and the seminal works of Klartag and Kolesnikov \cite{Klart,KlaKol,Kol2,Kol} seem to suggest that in the free context, and for a centered free Gibbs measure $\mu:=\nu_{u}$ associated with a smooth uniformly convex potential $u:\mathbb{R}\rightarrow \mathbb{R}$ (and thus supported on a compact interval $K$), the good notion of a non-commutative Hessian manifold (at least in dimension one) is given by the structure $M_{n.c}^*:=(K_{\varphi},\mathscr{J}\mathscr{D}\varphi, \nu_{\varphi})$ where $\varphi$ is the moment map and whose Laplacian can be computed explicitly and for which the moment map is always an eigenfunction associated with the eigenvalue $-1$. The pushforward of the free diffusion $(X_t)_{t\geq 0}$, by $\psi'$ where $\psi=\varphi^*$ can thus be interpreted as the (canonical) free Brownian motion on the non-commutative Hessian manifold $M_{n.c}:=(K,\mathscr{J}\mathscr{D}\psi,\mu)$.  Moreover, this formal construction can be seen as a canonical way to deform a free Brownian motion so that its invariant measure is exactly given by the free moment measure $\nu_{\varphi}$. This amounts in the end to adding a drift term to the generator of the free Brownian motion on the manifold $M_{n.c}:=(K,\mathscr{J}\mathscr{D}\psi,\nu_{\varphi})$  (we recall that the global and intrinsic construction via Dirichlet forms define the classical Brownian motion on a closed Riemannian manifold $(M,g)$ as the unique Markov process with half the Laplace-Beltrami operator as its infinitesimal generator). It is also interesting to note that in the classical case in the uniformly convex setting we have exponentially fast convergence to the equilibrium of the SDE \eqref{class}, a statement we didn't investigate here.
\bigbreak
The recent work of Fathi and Mikulincer \cite{FM} provides new insights into how to prove stability estimates for invariant measures of diffusions that may be far from Gaussian. As mentioned above, the (classical) SDE's considered in their paper are generally irregular and not necessarily Lipschitz, but generally belong to some Sobolev spaces. The first authors to overcome these problems and deal with this level of generality in the classical case was Figalli in \cite{Figal} for stochastic differential equations based on Le-Bris and Lions \cite{DLyo} work for Fokker-Plank equations with irregular coefficients. Unfortunately, in the free setting, almost nothing is known in the general context of free SDE's with non {\it operator Lipschitz} coefficients (except for the free SPDE's considered in the seminal work of Dabrowski \cite{DAB14}). It would be of great interest to develop a free analogue of these results, and also to investigate whether a kind of Lusin-Lipschitz property, i.e. maximal estimates for Sobolev functions w.r.t. the Lesbesgue measure, recently extended by Ambrosio, Bru\'{e} and Trevisan \cite{Lusin} when the reference measure is log-concave, might hold in the non-commutative context (see e.g. section 2.2 in \cite{FM} for precise statements).
\bigbreak
In the free context, the lack of geometric interpretations and a sufficiently strong free Riemannian geometry (in particular, it is closely related to understand the evolution of the semigroup acting on a Lipschitz function under the action of the non-commutative {\it carré du champ}, and therefore to prove the free log-Sobolev inequality under an appropriate non-commutative Bakry-Émery condition introduced by Guionnet and Shlyakhtenko in \cite{bak}) seems to be a major obstacle to obtaining such results. This question, as well as the study of the regularity properties of the free transport map, its $L^p$-estimates in some non-commutative Sobolev spaces, and the study of the continuity of the moment measure (most of these statements were proved by Klartag and Kolesnikov \cite{Klart,KlaKol, Kol2}), thus seems much more difficult to obtain in the free setting. Even extending our result in the multidimensional setting beyond the perturbative regime seems to be very complicated and possibly irrelevant. 
\bigbreak
The last question we ask is whether there is a possible free analogue of Fathi's conjecture $2$ in \cite{hdr}. A sub-question of this has already been raised in a previous author's paper \cite{Di}, section $10$, where a possible fourth moment diffusive structure was proposed, which would be a free analog of Ledoux's one. However, it is still very unclear whether there is any structure other than the semicircular one in which the free CLT can hold for random chaotic eigenfunctions. Moreover, the development of the theory of free moment maps in a convex non-perturbative setting could possibly be of great interest to translate some isomorphism problems. We leave all these questions to the interested reader.

\section{Acknowledgments}
CP Diez was partially supported by the Luxembourg National Research Fund
\newline
(Grant: O22/17372844/FraMStA). CP Diez also acknowledges support from  the Ministry of Research, Innovation and Digitalization (Romania), grant CF-194-PNRR-III-C9-2023. 
\newline
The author would like to thank  Pr. Max Fathi, Pr. Ivan Nourdin, Pr. Giovanni Peccati for their encouragement and insightful comments during the preparation of this work.

\end{document}